\documentclass[10pt,reqno]{amsart}
\usepackage{amsmath}
\usepackage{amsthm}
\usepackage{amssymb}
\usepackage{amsfonts}
\usepackage{subcaption}
\usepackage{graphicx}
\usepackage{latexsym}
\usepackage{cite, url}
\usepackage[dvipsnames]{xcolor}
\usepackage{stmaryrd}
\usepackage{hyperref}
\usepackage{enumitem}
    \setenumerate{itemsep=5pt} % adjust space between list items
    \setitemize{itemsep=5pt} % adjust space between list items
    
    \RequirePackage[dvipsnames]{xcolor}

\usepackage{mathtools}
\usepackage{amsbsy}    

\usepackage{enumitem}

\usepackage{pgfplots}

\usetikzlibrary{math}
\pgfplotsset{compat=1.16}

\usepackage[font=Large]{caption}

\usepackage[font=footnotesize,labelfont=bf]{caption}

\usepackage[sc]{mathpazo}
\usepackage[T1]{fontenc} 
\graphicspath{{IMAGES/}}
\newcommand{\Z}{\mathbb{Z}}
\newcommand{\R}{\mathbb{R}}
\newcommand{\M}{\mathsf{M}}

\newcommand{\T}{\!\mathsf{T}}

\renewcommand{\vec}[1]{{\bf #1}}
\providecommand{\multi}[1]{\llbracket #1 \rrbracket}

\renewcommand{\pmod}[1]{\,(\operatorname{mod} #1)}
\renewcommand{\phi}{\varphi}

\DeclarePairedDelimiter{\floor}{\lfloor}{\rfloor}
\DeclarePairedDelimiter{\ceil}{\lceil}{\rceil}

\newcommand{\norm}[1]{\| #1 \|}
\newcommand{\inner}[1]{\< #1 \>}

\newcommand{\ZZ}{\mathcal{Z}}
\newcommand{\ZZZ}{\mathsf{Z}}

\newcommand{\twovector}[2]{\begin{bmatrix} #1 \\ #2  \end{bmatrix}}

\renewcommand\>{\rangle}
\newcommand\<{\langle}

\newcommand{\abs}[1]{\left|#1\right|}

\renewcommand{\L}{\mathcal{L}}

\newcommand{\Mean}{\mathsf{Mean}\,}
\newcommand{\Median}{\mathsf{Median}\,}
\newcommand{\Mode}{\mathsf{Mode}\,}
\newcommand{\Var}{\mathsf{Var}\,}
\newcommand{\StDev}{\mathsf{StDev}\,}

\renewcommand{\min}{\mathsf{Min}\,}

\theoremstyle{plain}
\newtheorem{theorem}[equation]{Theorem}
\newtheorem{lemma}[equation]{Lemma}

\newtheorem{proposition}[equation]{Proposition}

\theoremstyle{definition}
\newtheorem{remark}[equation]{Remark}

\newtheorem{example}[equation]{Example}

\allowdisplaybreaks
\begin{document}
\title[Factorization length distribution for affine semigroups IV]{Factorization length distribution for affine semigroups~IV:\ a geometric approach to weighted factorization lengths in three-generator numerical semigroups}

\author[S.R.~Garcia]{Stephan Ramon Garcia}
\address{Department of Mathematics and Statistics, Pomona College, 610 N. College Ave., Claremont, CA 91711} 
\email{stephan.garcia@pomona.edu}
\urladdr{\url{http://pages.pomona.edu/~sg064747}}

\author[C.~O'Neill]{Christopher O'Neill}
\address{Mathematics Department, San Diego State University, 5500 Campanile Dr., San Diego, CA 92182}
\email{cdoneill@sdsu.edu}
\urladdr{\url{https://cdoneill.sdsu.edu/}}

\author[G.~Udell]{Gabe Udell}
\address{Department of Mathematics, Cornell University, 301 Tower Rd, Ithaca, NY 14853}
\email{gru5@cornell.edu}

\thanks{Partially supported by NSF grants DMS-1800123 and DMS-2054002.}

\begin{abstract}
For numerical semigroups with three generators, we study the asymptotic behavior of weighted factorization lengths, that is, linear functionals of the coefficients in the factorizations of semigroup elements.  This work generalizes many previous results, provides more natural and intuitive proofs, and yields a completely explicit error bound.
\end{abstract}


% Insert McNugget later on as a keyword...
\maketitle

%%%%%%%%%%%%%%%%%%%%%%%%%%%%%%%%%%%%%%%%%%%%%%%%%%%%%%
\section{Introduction}
In what follows, $\Z_{\geq 0}$ and $\Z_{>0}$ denote the sets of nonnegative and positive integers, respectively.
Let
\begin{equation*}
S = \inner{n_1,n_2,n_3}= \{ x_1 n_1 + x_2 n_2 + x_3 n_3 : x_1,x_2,x_3\in \Z_{\geq 0}\}
\end{equation*}
denote a \emph{numerical semigroup} (an additive subsemigroup of $\Z_{\geq 0}$) with three \emph{generators}
$n_1,n_2,n_3 \in \Z_{>0}$ \cite{MR2549780}.  
We do not assume that the generators are given in a particular order; on rare occasions, we even let them coincide.
Although unconventional, these generous conventions eliminate the need for some special cases and permit a few interesting
and unusual applications.

A \emph{factorization} of $n \in S$ is an expression 
$n = x_1n_1 + x_2 n_2 + x_3 n_3$ in which $\vec{x} = (x_1,x_2,x_3) \in \Z_{\geq 0}^3$.
The set of all factorizations of $n$ is denoted 
\begin{equation*}
\ZZZ_S(n) = \{ \vec x \in \Z_{\geq 0}^3 : n = x_1n_1 + x_2 n_2 + x_3 n_3 \}.
\end{equation*}
A \emph{factorization functional} is a linear functional of $\vec{x}$.
For example, the \emph{length} $x_1 + x_2 + x_3$ of $\vec{x}$ is a factorization functional.  
Other examples are $x_1$ and $x_1 - 2x_2 + 3x_3$.   
Values of factorization functionals are 
\emph{weighted} factorization lengths.  Combinatorial descriptions of maximum and minimum weighted factorization lengths are obtained in \cite{MR4142067}; the present paper pushes this work in new directions.

In this paper, which complements the previous papers in the series~\cite{GOY, Modular, GOOY}, we answer many questions about the asymptotic behavior
of weighted factorization lengths for three-generator numerical semigroups.
We recover and extend~\cite{GOY}, in which the
asymptotic behavior of the mean, median, and mode of (unweighted) factorization lengths are described.
Our Theorem~\ref{Theorem:Main} is more general and 
more precise than the main results of \cite{GOY}, and its proof is shorter and more transparent.

The paper \cite{GOOY}, which subsumes \cite{GOY}, treats numerical semigroups with arbitrarily many generators.
However, the approach uses tools outside the mainstream of numerical semigroup theory, such as 
algebraic combinatorics, harmonic analysis, measure theory, and functional analysis.  
In contrast, our results here are geometric and transparent, only invoking
analysis (of an elementary sort) at the final stage.  The results of this paper,
although presented only for three-generator numerical semigroups (but for general weighted factorization lengths),
may provide a clearer path to the results of \cite{GOOY} and their generalizations via polyhedral geometry techniques.

As convenience dictates, we denote (column) vectors in boldface, or as ordered pairs or triples. 
A superscript $\T$ denotes the transpose.  We let $|X|$ denote the cardinality of a set or multiset $X$.
Here is our main result.

\begin{theorem}\label{Theorem:Main}
Let $n_1,n_2,n_3 \in \Z_{> 0}$ be distinct with $\gcd(n_1,n_2,n_3)=1$; let $m_1,m_2,m_3 \in \Z$ be such that
\begin{equation*}
 \frac{m_3}{n_3} \,\,\leq\,\, \frac{m_2}{n_2} \,\,\leq\,\, \frac{m_1}{n_1},
\end{equation*}
with at least one inequality strict; let
$S = \inner{n_1,n_2,n_3}$; and let
\begin{equation*}
\lambda(\vec{x}) = m_1 x_1 + m_2 x_2 + m_3 x_3
\end{equation*}
for $\vec{x} \in \Z^3$.  Define the multiset (set with multiplicities)
\begin{equation*}
\Lambda\multi{n} = \{\!\!\{ \lambda(\vec{x}) : \vec{x} \in \ZZZ_S(n) \}\!\!\}.
\end{equation*}
Then for $\alpha < \beta$ and $n \in \Z_{\geq 0}$,
\begin{align*}
&\left|\frac{ \big| \Lambda \multi{n} \cap  [\alpha n ,\beta n] \big| }{n^2/(2n_1n_2n_3)} -\int_{\alpha}^\beta  F(x)\,dx \right|  \\
&\qquad\qquad \leq \frac{2n_1n_2n_3}{n} \left[ \frac{5d }{n_2}+\frac{2d}{n} + \left(\beta - \alpha+\frac{2d}{n}\right)\left(1+ d\operatorname{max}\left\{ n_1,n_3 \right\} \right) \right],
\end{align*}
in which
\begin{equation*}
F(t) = \frac{2 n_1 n_2 n_3}{m_1 n_3 - m_3 n_1}
\begin{cases}
0 & \text{if $t < \frac{m_3}{n_3}$},\\[3pt]
\dfrac{t n_3 - m_3 }{ m_2 n_3 - m_3 n_2,} \qquad& \text{if $\frac{m_3}{n_3} \leq t < \frac{m_2}{n_2}$},\\[10pt]
 \dfrac{m_1 - n_1 t }{ m_1 n_2 - m_2 n_1 } & \text{if $\frac{m_2}{n_2} \leq t < \frac{m_1}{n_1}$},\\[5pt]
0 & \text{if $t \geq \frac{m_1}{n_1}$},
\end{cases}
\end{equation*}
is a (possibly degenerate) triangular probability density function, and 
\begin{equation*}
d=\gcd(m_2n_3-m_3n_2, \, m_1n_3-m_3n_1, \, m_1n_2-m_2n_1). 
\end{equation*}
\end{theorem}

The rate of convergence to the triangular density is explicit,
a huge improvement over \cite{GOY}.
Moreover, the error estimate can be improved at the expense of introducing a more complicated, 
but still explicit, expression; see Remark~\ref{Remark:Error}.
Modifications of our Lemma \ref{Lemma:Convergence} below should also permit us to recover the modular results of \cite{Modular} in the three-generator setting,
with the added bonus of explicit bounds on the rate of convergence in \cite[Thm.~3a]{Modular}.  

The motivation for Theorem \ref{Theorem:Main} stems from its centrality to the study of numerical semigroups.
Non-unique factorization has long been studied in commutative algebra, both for more general families of semigroups \cite{schmid2009characterization,geroldinger2006non,chapman2000half,nonuniq,
abhyankar1967local,chapman2014delta} and for numerical semigroups specifically \cite{numericalfactorsurvey,nselasticities,barucci1997maximality,numericalrealization}.  
The study of length sets (as opposed to multisets) is well-established territory
\cite{GH92,structurethm,narkiewiczconjecture,gao2000systems,setsoflength} and similar questions
have been studied in both number-theoretic~\cite{integervaluedpolys,factoralgebraicintegers,acmfirst}
and algebraic~\cite{modulessurvey,noncommutativefactor,setsoflengthmonthly} contexts.
Our explicit asymptotic theorem on weighted factorization lengths and multisets breaks 
new ground in the three-generator setting.

This paper is structured as follows.
We first consider examples and applications in Section~\ref{Section:Examples},
after which we move into the proof of Theorem~\ref{Theorem:Main} in Section~\ref{Section:Proof}.

%%%%%%%%%%%%%%%%%%%%%%%%%%%%%%%%%%%%%%%%%%%%%%%%%%%%%%

\section{Examples and Applications}\label{Section:Examples}

Throughout this section we consider pairs of vectors $\vec{m}=(m_1,m_2,m_3) \in \Z^3$ 
and $\vec{n}=(n_1,n_2,n_3)\in \Z_{>0}^3$ which satisfy the conditions of Theorem \ref{Theorem:Main}. 
In each such context we define $S=\langle n_1, n_2, n_3 \rangle$,
\begin{equation*}
\lambda(\vec{x})=m_1x_1+m_2x_2 +m_3x_3,
\end{equation*}
and 
\begin{equation*}
\Lambda\multi{n} = \{\!\!\{ \lambda(\vec{x}) : \vec{x} \in \ZZZ_S(n) \}\!\!\}
\end{equation*} 
as in the statement of Theorem \ref{Theorem:Main}.
We also define 
\begin{equation}\label{eq:ZZb}
    \ZZ(m, n) = \{ \vec{x} \in \ZZZ_S(n) : \lambda(\vec{x})=m\}.
\end{equation}

Our first application of Theorem \ref{Theorem:Main} is to swiftly obtain general 
weighted versions of the main results of \cite{GOY}, in 
which the asymptotic mean, median, and mode (unweighted) factorization lengths are computed
for three-generator numerical semigroups.
In what follows, $f \sim g$ means that $\lim_{n\to\infty} f(n)/g(n) = 1$.

\begin{example}
Let $S = \inner{n_1,n_2,n_3}$, in which $\gcd(n_1,n_2,n_3)=1$. 
Apply Theorem \ref{Theorem:Main} with $\alpha = \frac{m_3}{n_3}$ and $\beta = \frac{m_1}{n_1}$ and obtain
\cite[Thm.~3.9]{factorhilbert}:
\begin{equation}\label{eq:ZZZOn}
| \Lambda(n) | = | \ZZZ_S(n)| \,\,\sim\,\, \frac{n^2}{2n_1n_2n_3} .
\end{equation}
For $\alpha < \beta$, Theorem \ref{Theorem:Main} and \eqref{eq:ZZZOn} ensure that
\begin{equation*}
\frac{ \big| \Lambda\multi{n} \cap  [\alpha n ,\beta n]  \big| }{| \Lambda\multi{n}| }
\,\,\sim\,\, \int_{\alpha}^\beta  F(x)\,dx
\end{equation*}
as $n \to \infty$.  Since the support of $F$ is $[\frac{m_3}{n_3}, \frac{m_1}{n_1}]$ and its
peak is at $\frac{m_2}{n_2}$, we have
\begin{equation*}
\min \Lambda\multi{n}  \,\, \sim \,\, \frac{m_3}{n_3} n,\qquad 
\Mode \Lambda\multi{n}  \,\, \sim \,\, \frac{m_2}{n_2} n,
\qquad \text{and} \qquad
\operatorname{max} \Lambda\multi{n}  \,\, \sim \,\, \frac{m_1}{n_1} n.
\end{equation*}
Symbolic integration and computer algebra reveals the unique
$\gamma \in [\frac{m_3}{n_3}, \frac{m_1}{n_1}]$ such that $\int_{-\infty}^{\gamma} F(t)\,dt = \frac{1}{2}$.
This yields the asymptotic median:
\begin{equation*}
\Median \Lambda\multi{n} \,\,\sim\,\, 
n \cdot \small
\begin{cases}\displaystyle
\frac{m_3}{n_3} + \sqrt{ \frac{1}{2} \left( \frac{m_1}{n_1} - \frac{m_3}{n_3}\right)\left( \frac{m_2}{n_2} - \frac{m_3}{n_3} \right)}
&\displaystyle \text{if $\frac{m_2}{n_2}  \geq \frac{1}{2} \left( \frac{m_1}{n_1} + \frac{m_3}{n_3} \right)$},\\[10pt]
\displaystyle\frac{m_1}{n_1} - \sqrt{ \frac{1}{2} \left( \frac{m_1}{n_1} - \frac{m_3}{n_3}\right)\left( \frac{m_1}{n_1} - \frac{m_2}{n_2} \right)}
&\displaystyle \text{if $\frac{m_2}{n_2} < \frac{1}{2} \left( \frac{m_1}{n_1} + \frac{m_3}{n_3} \right)$}.
\end{cases}
\end{equation*}

Consider the absolutely continuous probability measure $\nu$ defined by
\begin{equation*}
\nu( [\alpha,\beta] ) = \int_{\alpha}^{\beta} F(x)\,dx
\end{equation*}
for $\alpha < \beta$.
Define the singular probability measures
\begin{equation*}
\nu_n = \frac{1}{| \ZZZ_S(n)|} \sum_{ \vec{x} \in \ZZZ_S(n) } \delta_{ \frac{\mu(\vec{x})}{n}} ,
\end{equation*}
in which $\delta_x$ is the unit point measure at $x \in \R$.
Use \eqref{eq:ZZZOn} to deduce that
\begin{equation*}
\lim_{n\to\infty} \nu_n( [\alpha,\beta])
= \lim_{n\to\infty}\frac{\big| \Lambda(n) \cap  [\alpha n ,\beta n]  \big|  }{ | \ZZZ_S(n)| }
= \int_{\alpha}^\beta  F(x)\,dx = \nu( [ \alpha, \beta]).
\end{equation*}
If $g:\R\to\R$ is bounded and continuous, then \cite[Thm.~25.8]{Billingsley} ensures that
\begin{equation*}
\lim_{n\to\infty} 
 \frac{1}{| \Lambda(n)|} \sum_{ \vec{x} \in \ZZZ_S(n) } g\bigg( \frac{\lambda(\vec{x})}{n} \bigg) 
= \lim_{n\to\infty} \int_{\R} g \,d\nu_n
=  \int_{\R}  g(x) F(x)\,dx.
\end{equation*}
The integral on the right-hand side 
can be evaluated explicitly for $g(x) = x$ and $g(x)=x^2$.  From here one obtains
the asymptotic mean and variance of $\Lambda\multi{n}$:
\begin{align*}
\Mean \Lambda\multi{n}&\,\,\sim\,\, \frac{n}{3}\left( \frac{m_1}{n_1} + \frac{m_2}{n_2} + \frac{n_3}{n_3} \right)  ,\\[5pt]
\Var \Lambda\multi{n} &\,\,\sim\,\, \frac{n^2}{18}\left( \frac{m_1^2}{n_1^2} + \frac{m_2^2}{n_2^2} + \frac{m_3^2}{n_3^2} 
- \frac{m_1 m_2}{n_1 n_2} - \frac{m_2 m_3 }{n_2 n_3}- \frac{m_3 m_1}{n_3 n_1} \right).
\end{align*}
Asymptotic formulas for the higher moments, skewness, harmonic and geometric means,  follow in a
similar manner; see \cite[Sec.~2.1]{GOOY} for definitions.
For $m_1 = m_2 = m_3 = 1$, we obtain the asymptotic formulas
for factorization-length statistics obtained in \cite{GOY}.  Thus, Theorem \ref{Theorem:Main} recaptures
the results of \cite{GOY}, generalizes them to the weighted setting, and provides explicit error bounds in some instances.
\end{example}

\begin{example}
In \cite[Tab.~1, Fig.~2]{GOOY}, a special case of Theorem \ref{Theorem:Main} was illustrated for factorization lengths in 
the McNugget semigroup $S=\langle 6,9,20\rangle$.
Here we explore a different weighted factorization length on $S$.
Table \ref{Table:statstable} gives the actual and predicted values of several statistics pertaining to $\Lambda\multi{n}$ for $\vec{m}=(4,7,2)$, $\vec{n}=(9,20,6)$, and $n=10^5$.
The components of $\vec{m}$ and $\vec{n}$ are ordered to comply with Theorem \ref{Theorem:Main}; in particular $4/9>7/20>2/6$.  
If one charges $\$2$ for a box of $6$ McNuggets, $\$4$ for $9$ McNuggets, and $\$7$ for $20$ McNuggets, then
$\Lambda\multi{n}$ is the multiset of prices corresponding to all the ways to fill an order of $n$ McNuggets. 
\end{example}

\begin{table}[h]
\begin{equation*}
\begin{array}{c|cc||c|cc}
\text{Statistic} & \text{Actual} & \text{Predicted} & \text{Statistic} & \text{Actual} & \text{Predicted} \\
\hline
\Mean \Lambda\multi{10^5} & 37591.84  & 37592.59 & \Mode\Lambda\multi{10^5} & 35000 & 35000  \\[3pt]
\Median \Lambda\multi{10^5} & 37200  & 37200.89  & \StDev\Lambda\multi{10^5} & 2446.32 &  2446.27 \\[3pt]
\min \Lambda\multi{10^5} & 33334 & 33333.33
& \operatorname{max} \Lambda \multi{10^5} & 44440 & 44444.44
\end{array}
\end{equation*}
\caption{Actual versus predicted statistics (rounded to two decimal places) for $\Lambda\multi{10^5}$ with $\vec{n}=(9,20,6)$ and $\vec{m}=(4,7,2)$}
\label{Table:statstable}
\end{table}

%%%%%%%%%%%%%%%%%%%%%%%

The next example illustrates another use of Theorem \ref{Theorem:Main}.

\begin{example}
Let $S=\langle 6,9,20\rangle$ as in the previous example.
We now let $\vec{n}=(1,1,1)$ and $\vec{m}=(20,9,6)$. 
Then $\big|\Lambda\multi{n}\cap [\alpha n,\beta n]\big|$ is the number of possible orders of $n$ boxes of McNuggets that contain between $\alpha n$ and $\beta n$ McNuggets. 
For example, when $n=100$, $\alpha=8$, and $\beta=15$, we have 
$\big|\Lambda[\![100]\!]\cap [800, 1500] \big|=3785$;
that is, there are $3785$ ways to order between $800$ and $1500$ McNuggets using
$100$ boxes. 
Table \ref{Table:errortab} illustrates predictions and error bounds afforded by Theorem \ref{Theorem:Main}
and \eqref{eq:Error}.
\end{example}

\begin{table}[h]\footnotesize
\begin{equation*}
\begin{array}{ccc|ccccc}
n &\alpha & \beta & \frac{|\Lambda[\![n]\!]\cap [\alpha n, \beta n]|}{n^2/2} &\int_{\alpha}^{\beta}F(x)\,dx & \text{Error} & \text{Theorem \ref{Theorem:Main} bound} & \text{Eq.~\ref{eq:Error} bound} \\
\hline
100 & 8 & 15 & 0.757     & 0.742424 & 0.014576& 0.3812 & 0.151286 \\
1000 & 8 & 15 & 0.743884 & 0.742424 & 0.001460 & 0.038012 & 0.015056\\
10000 & 8 & 15 & 0.742570 & 0.742424 & 0.000146 & 0.003800 & 0.001505\\
100 & 7 & 7.1 & 0.0058 & 0.005 & 0.0008 & 0.1052 & 0.01\\
1000 & 7 & 7.1 & 0.00509 & 0.005 & 0.00009 & 0.010412 & 0.000927\\
10000 & 7 & 7.1 & 0.005009 & 0.005 & 0.000009 & 0.001040 & 0.000092
\end{array}
\end{equation*}
\caption{Error analysis (rounded to $6$ decimal places) 
for $\vec{m}=(20,9,6)$ and $\vec{n}=(1,1,1)$.}
\label{Table:errortab}
\end{table}

In the following examples, we plot $\frac{|\ZZ(m,n)|}{dn/(2n_1n_2n_3)}$ versus $\frac{m}{n}$ (in blue) overlaid 
with $F(x)$ versus $x$ (in red). 
These make sense to plot together because Lemma~\ref{Lemma:LineCount} and equation \ref{eq:Fell} below 
imply that $\frac{|\ZZ(m,n)|}{dn/(2n_1n_2n_3)}$ is within $\frac{2n_1n_2n_3}{dn}$ of $F(\frac{m}{n})$. 
Since $|\ZZ(m,n)|$ gives the multiplicity of $m$ in $\Lambda\multi{n}$, we refer to this sort of plot as 
the \emph{scaled histogram} of $\Lambda\multi{n}$.
These plots illustrate the convergence of the distribution of $\Lambda\multi{n}$ to $F(x)$.

\begin{example}
Figure \ref{Figure:revnug} gives the scaled histograms of $\Lambda\multi{100}$ and $\Lambda\multi{1000}$ 
for $\vec{m}=(20, 9, 6)$ and $\vec{n}=(1, 1, 1)$. 
\end{example}

\begin{figure}[h]
\centering
		\begin{subfigure}[b]{0.475\textwidth}
	                \centering
	                \includegraphics[width=\textwidth]{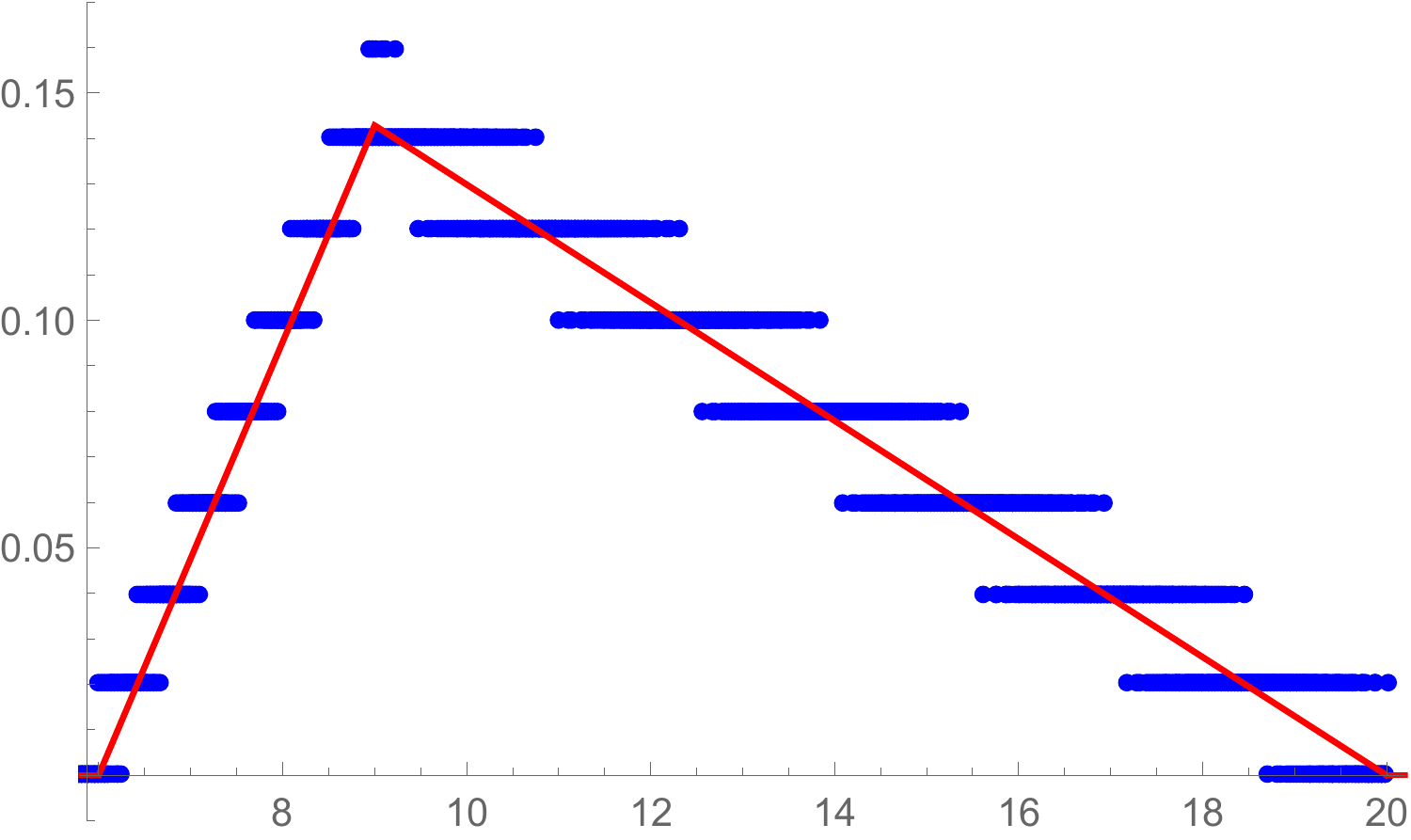}
	                \caption{$n=100$}
	        \end{subfigure}
	        \quad
		\begin{subfigure}[b]{0.475\textwidth}
	                \centering
	                \includegraphics[width=\textwidth]{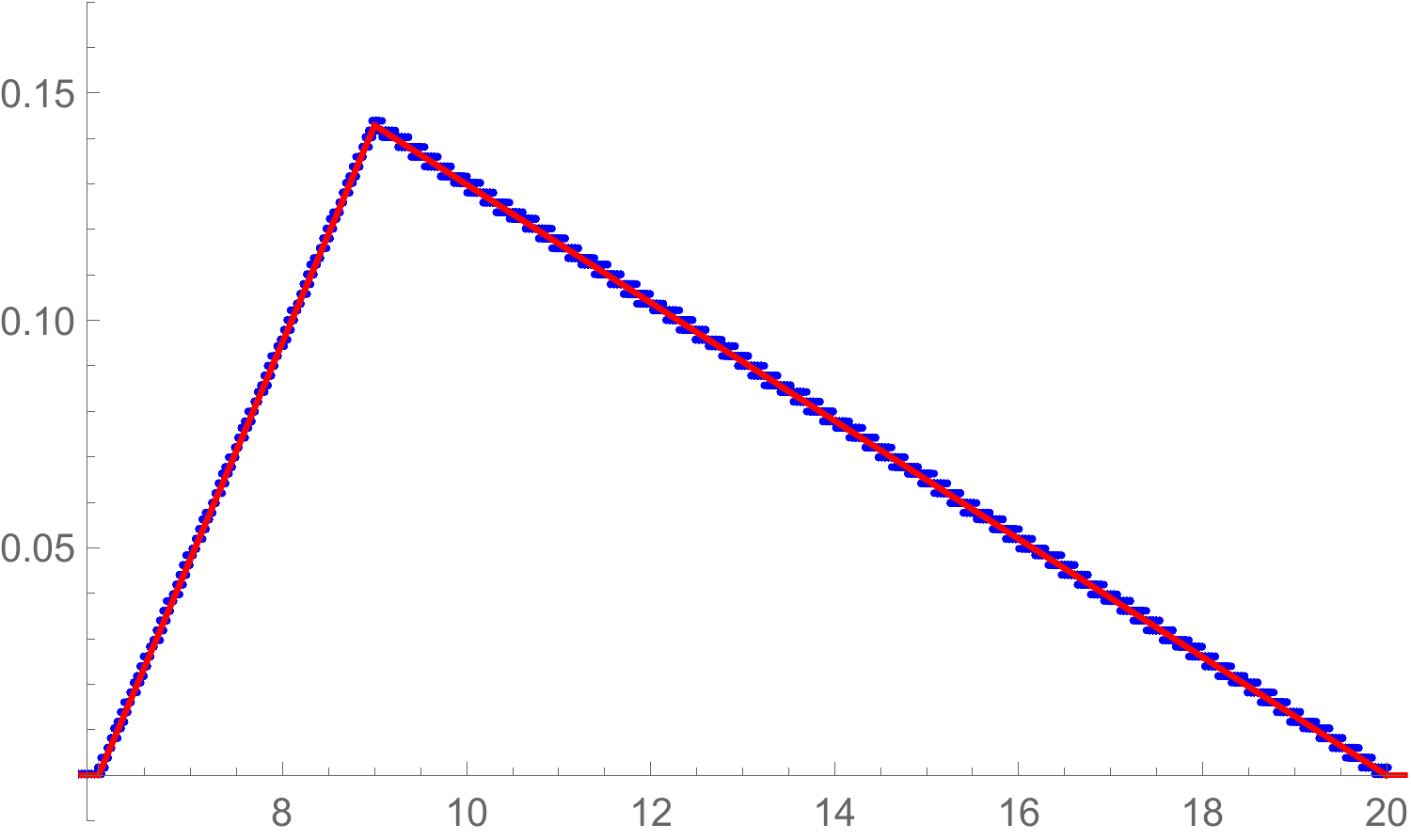}
	                \caption{$n=50000$}
	        \end{subfigure}

\caption{Scaled histograms of $\Lambda\multi{n}$ with $\vec{n}=(1,1,1)$ and $\vec{m}=(20,9,6)$.}
\label{Figure:revnug}
\end{figure}
%%%%%%%%%%%%%%%%%%%%%%%%

\begin{example}
Theorem \ref{Theorem:Main} does not require $m_1,m_2,m_3$ to be positive. 
Figure \ref{Figure:negex} demonstrates the theorem when $m_1<0$.
\end{example}

\begin{figure}[h]
\centering
		\begin{subfigure}[b]{0.475\textwidth}
	                \centering
	                \includegraphics[width=\textwidth]{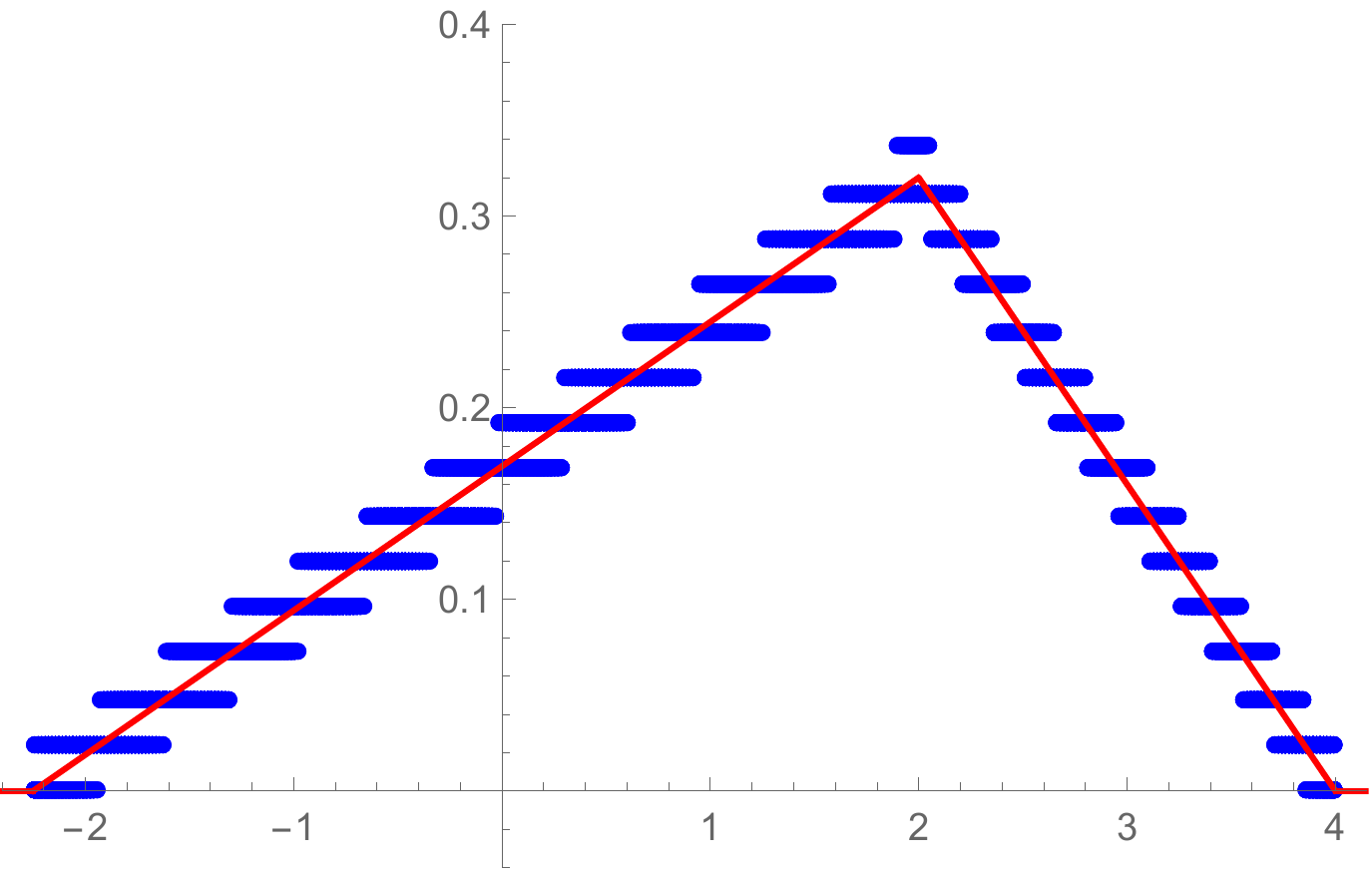}
	                \caption{$n=5000$}
	        \end{subfigure}
	        \quad
		\begin{subfigure}[b]{0.475\textwidth}
	                \centering
	                \includegraphics[width=\textwidth]{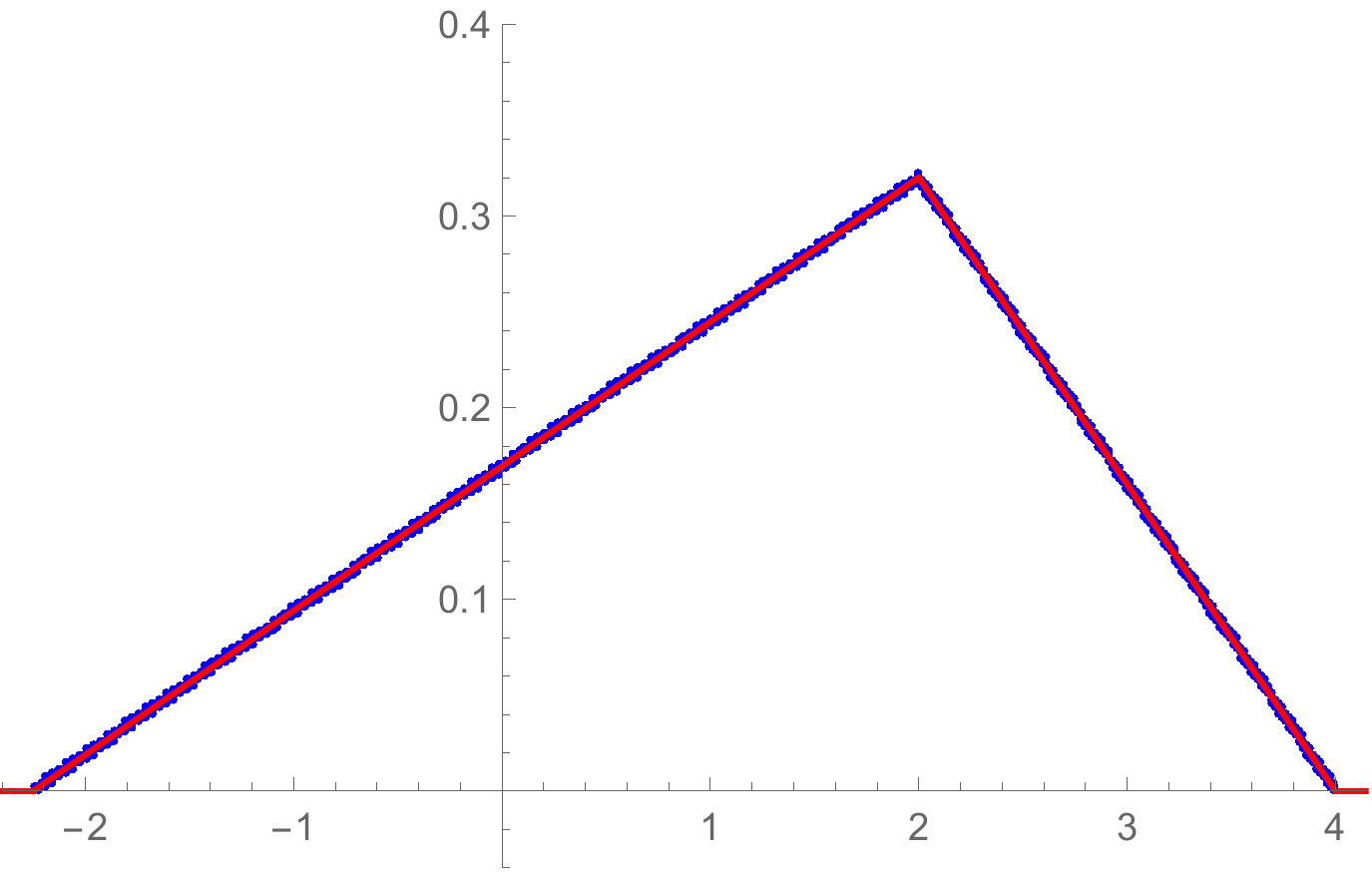}
	                \caption{$n=50000$}
	        \end{subfigure}

\caption{Scaled histograms of $\Lambda\multi{n}$ with  $\vec{m}=(-9,20,6)$ and $\vec{n}=(4,5,3)$.}
\label{Figure:negex}
\end{figure}
%%%%%%%%%%%%%%%%%%%%%%%%

\begin{example}
The error bound in Theorem \ref{Theorem:Main} and the definition of the scaled histogram involve the quantity $d=\gcd(m_2n_3-m_3n_2, m_1n_3-m_3n_1, m_1n_2-m_2n_1)$. For $d=1$, the scaled histogram of $\Lambda\multi{n}$ approximately coincides with the plot of $F(x)$ at each point. 
For $d\neq 1$, Lemma \ref{Lemma:whenEmpty} says that there is a $c = c_n$ such that $\ZZ(m,n)$ is empty unless
$m\equiv c \pmod{d}$. 
If $\ZZ(m,n)$ is nonempty, Lemma \ref{Lemma:LineCount} implies that its cardinality is $d$ times larger than what we would expect for $d=1$. This is accounted for in the definition of the scaled histogram so that $d-1$ out of every $d$ points of the scaled histogram of $\Lambda\multi{n}$ are $0$, but the remaining points approximately lie on the plot of $F(x)$;
see Figure \ref{Figure:gaps}.

\begin{figure}
\centering
		\begin{subfigure}[b]{0.475\textwidth}
	                \centering
	                \includegraphics[width=\textwidth]{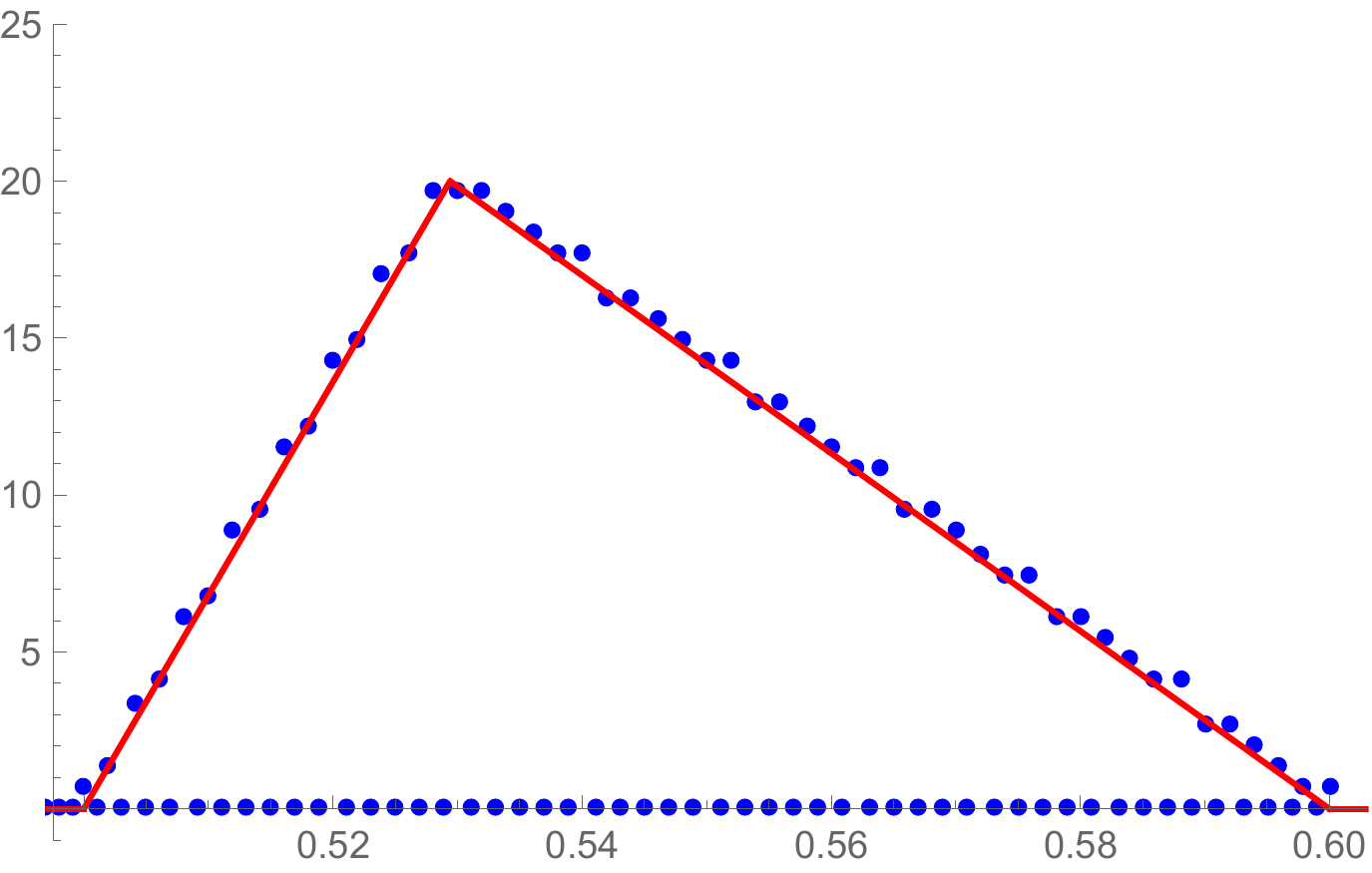}
	                \caption{$n=1000$}
	        \end{subfigure}
	        \quad
		\begin{subfigure}[b]{0.475\textwidth}
	                \centering
	                \includegraphics[width=\textwidth]{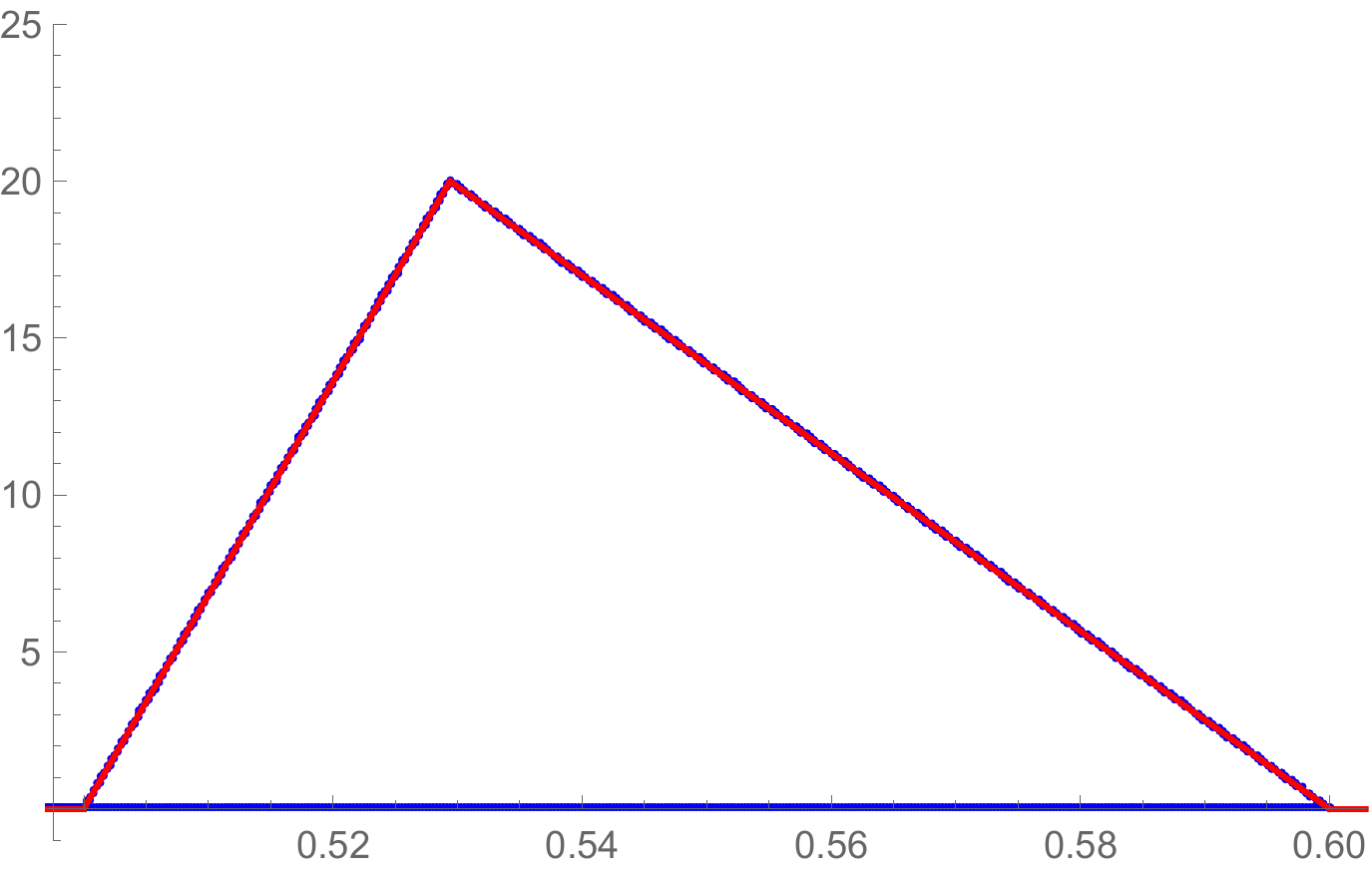} 
	                \caption{$n=30000$}\label{Figure:gapsfilled}
	        \end{subfigure}

\caption{Scaled histograms of $\Lambda\multi{n}$ with $\vec{n}=(5,17,8)$ and $\vec{m}=(3,9,4)$; here $d=2$.}
\label{Figure:gaps}
\end{figure}

\end{example}

%%%%%%%%%%%%%%%%%%%%%%%%

\begin{example}
The proof of Theorem \ref{Theorem:Main} defines
$\rho_1=m_2n_3-m_3n_2$ and $\rho_3=m_1n_2-m_2n_1$. 
Although these are denominators in the formula for $F$, we permit one of them to be $0$.
Figure \ref{Figure:isolate} illustrates the case $\vec{n}=(6,9,20)$ and $\vec{m}=(1,0,0)$, for which $\rho_1=0$. 
Here $\lambda(\vec{x}) = x_1$ is the number of $6$s in the factorization $6x_1 + 9x_2 + 30x_3 = n$.
Since $\rho_1=0$, the ``left side'' of the triangle is degenerate.
\end{example}

\begin{figure}[h]
\centering
		\begin{subfigure}[b]{0.475\textwidth}
	                \centering
	                \includegraphics[width=\textwidth]{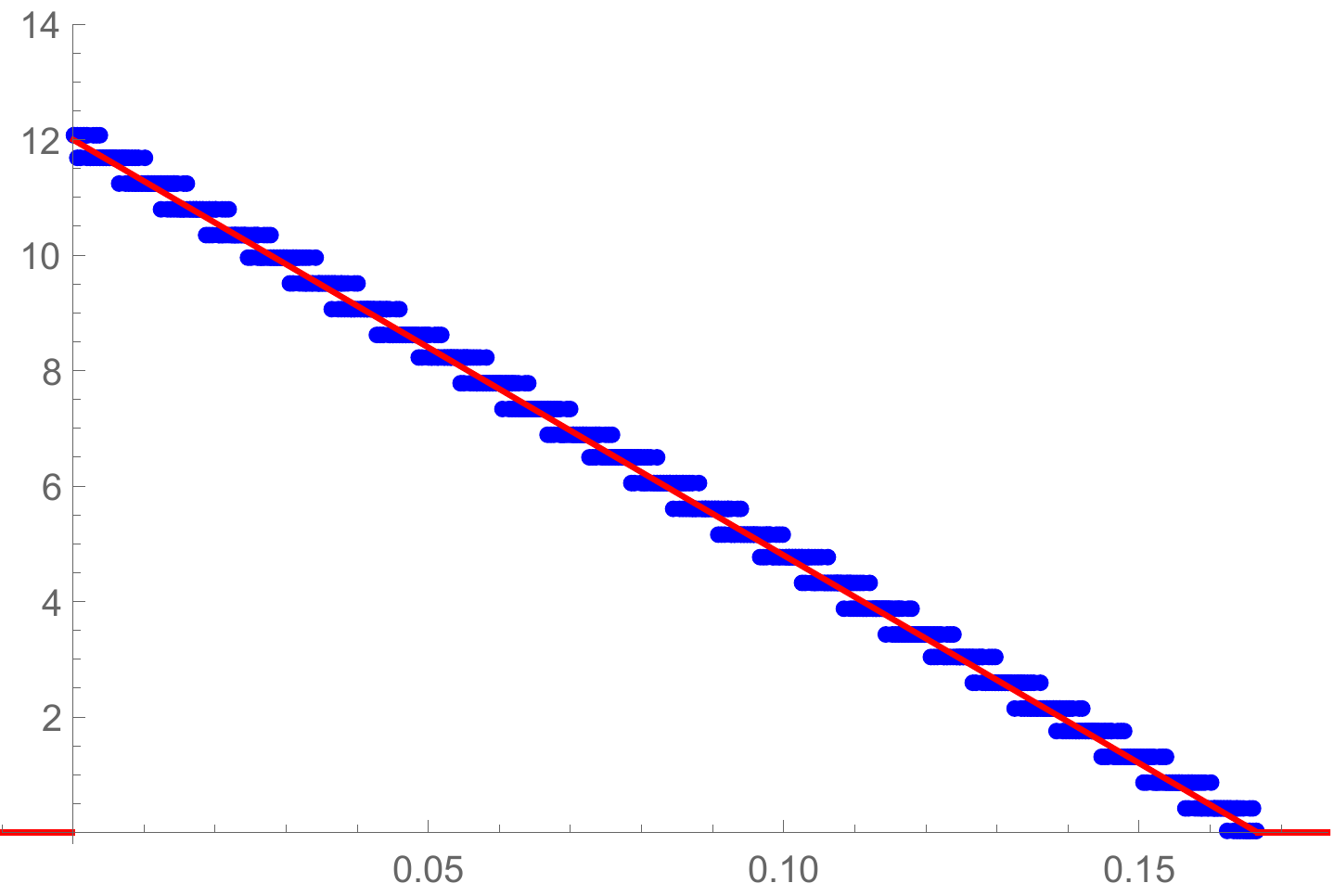}
	                \caption{$n=5000$}
	        \end{subfigure}
	        \quad
		\begin{subfigure}[b]{0.475\textwidth}
	                \centering
	                \includegraphics[width=\textwidth]{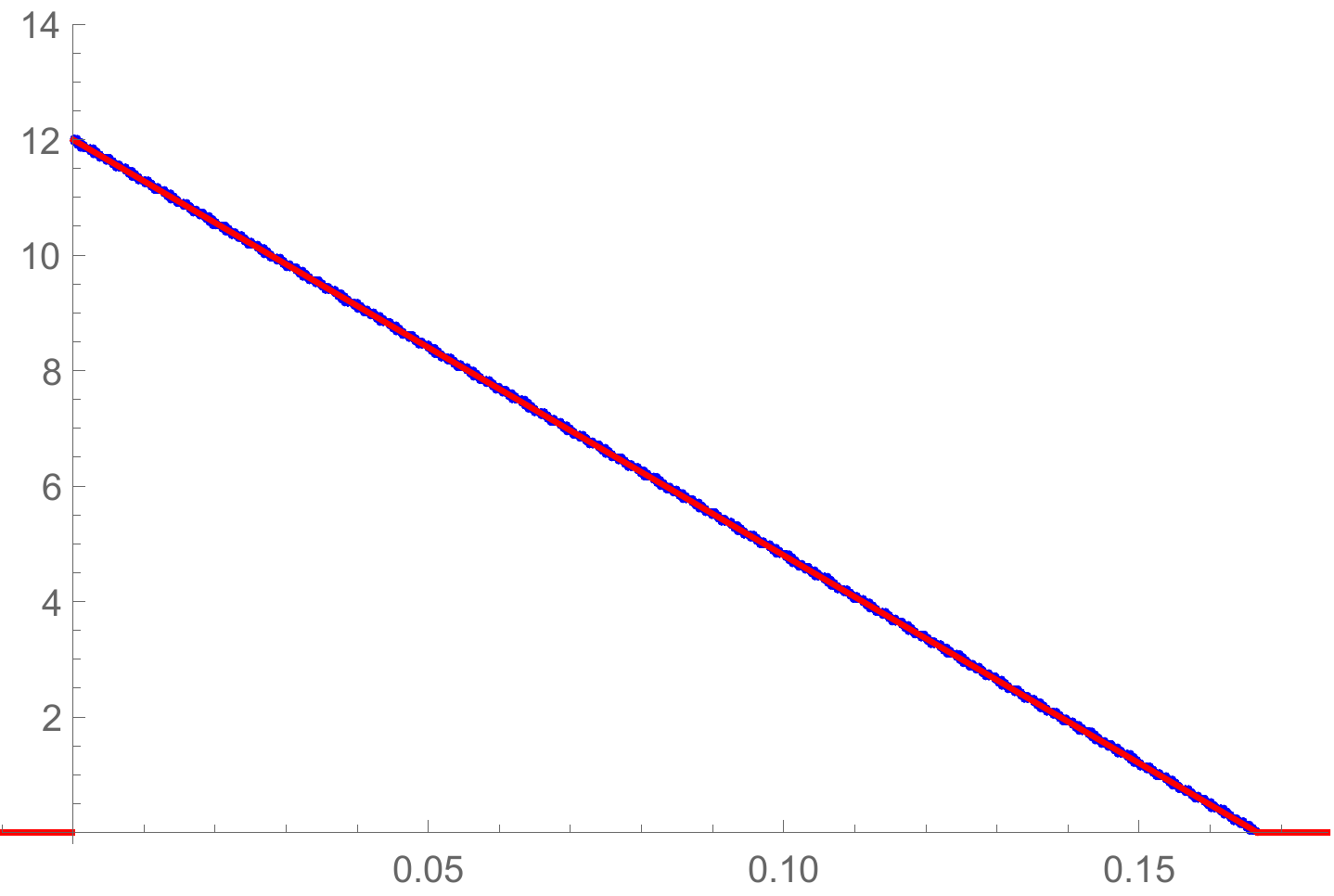}
	                \caption{$n=50000$}
	        \end{subfigure}

\caption{Scaled histograms of $\Lambda\multi{n}$ with $\vec{n}=(6,9,20)$ and $\vec{m}=(1,0,0)$.}
\label{Figure:isolate}
\end{figure}

%%%%%%%%%%%%%%%%%%%%%%%%

Theorem \ref{Theorem:Main} concerns large-$n$ asymptotic behavior.
On the other hand, Proposition~\ref{Proposition:supersymmetric} below 
identifies a curious exact phenomenon even for small $n$.  We first illustrate this with an example.

\begin{example}
Let $\vec{m}_1=(2,3,1)$, $\vec{n}_1=(2,6,3)$, $\vec{m}_2=(3,1,2)$, and $\vec{n}_2=(3,2,6)$;
note that $\vec{n}_1$ and $\vec{n}_2$ generate the same semigroup.
Figure~\ref{Figure:supersymmetric} shows the scaled histograms of the multisets $\Lambda_1\multi{n}$ and
$\Lambda_2\multi{n}$ corresponding to $\vec{m}_1,\vec{n}_1$ and to $\vec{m}_2,\vec{n}_2$, respectively.
The histograms are the same up to a horizontal translation.
To be specific, there is an $r$, which depends depends only upon $n$, such that the multiplicity of $x$ in $\Lambda_1\multi{n}$ equals 
the multiplicity of $x+r$ in $\Lambda_2\multi{n}$. In Figure~\ref{Figure:supersymmetric}, we have $n=75$ and $r=2$. 
Observe that the probability density $F$ depends only upon $m_1/n_1$, $m_2/n_2$, and $m_3/n_2$,
so Theorem~\ref{Theorem:Main} predicts the same asymptotic distribution for $\Lambda_1\multi{n}$ and $\Lambda_2\multi{n}$ because 
\begin{equation*}
\frac{m_1}{n_1}=\frac{2}{2}=\frac{m_1'}{n_1'}=\frac{3}{3}=1,
\qquad
\frac{m_2}{n_2}=\frac{3}{6}=\frac{m_2'}{n_2'}=\frac{1}{2},
\quad \text{and} \quad
\frac{m_3}{n_3}=\frac{1}{3}=\frac{m_3'}{n_3'}=\frac{2}{6},
\end{equation*}
However, this only implies that $\Lambda_1\multi{n}$ and $\Lambda_2\multi{n}$ should appear similar for large~$n$,
not that they should be translations of each other.
\end{example}

\begin{figure}
\centering
		\begin{subfigure}[b]{0.475\textwidth}
	                \centering
	                \includegraphics[width=\textwidth]{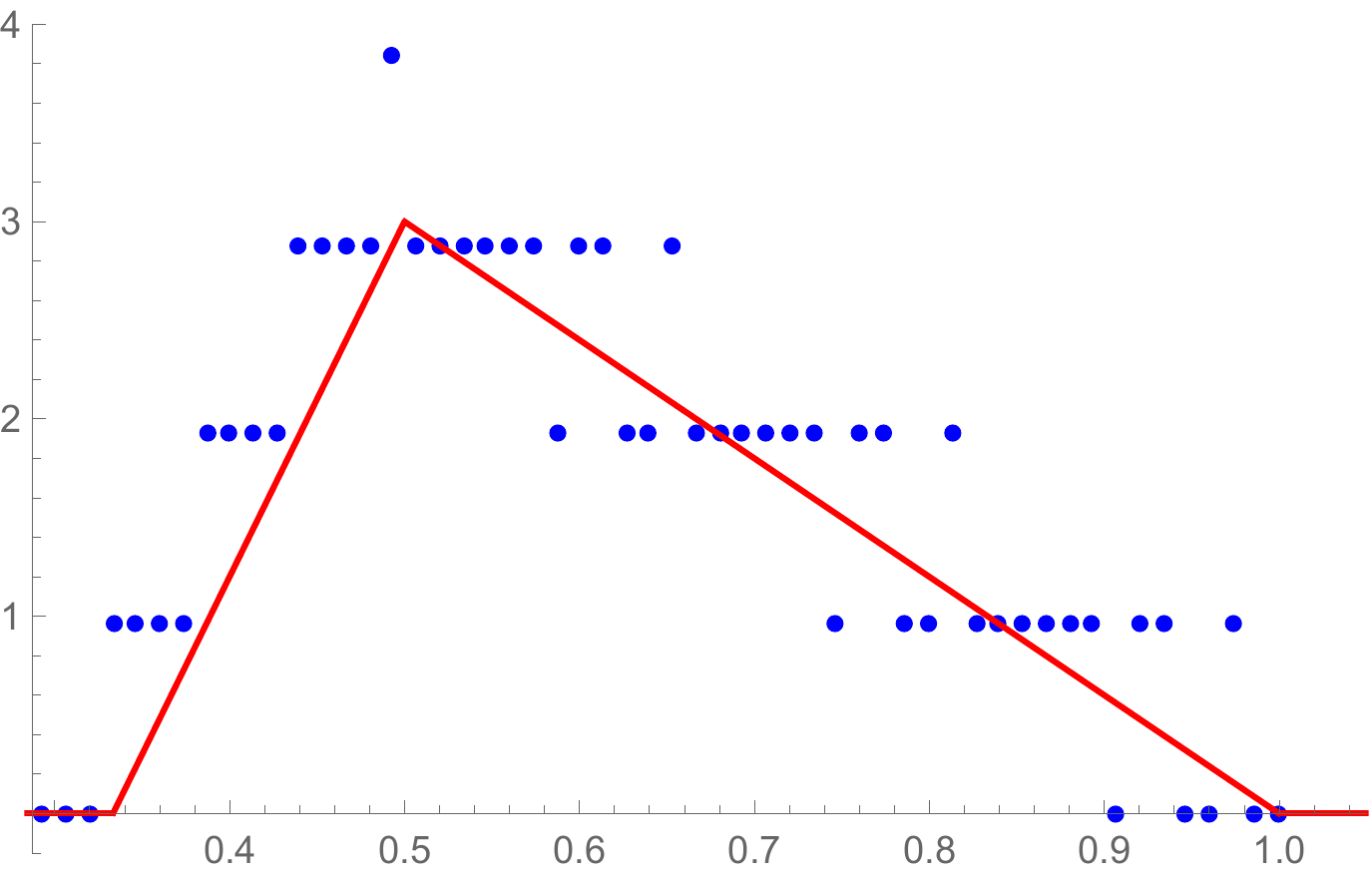}
	                \caption{$\vec{m}_1=(2,3,1)$, $\vec{n}_2=(2,6,3)$}
	        \end{subfigure}
	        \quad
		\begin{subfigure}[b]{0.475\textwidth}
	                \centering
	                \includegraphics[width=\textwidth]{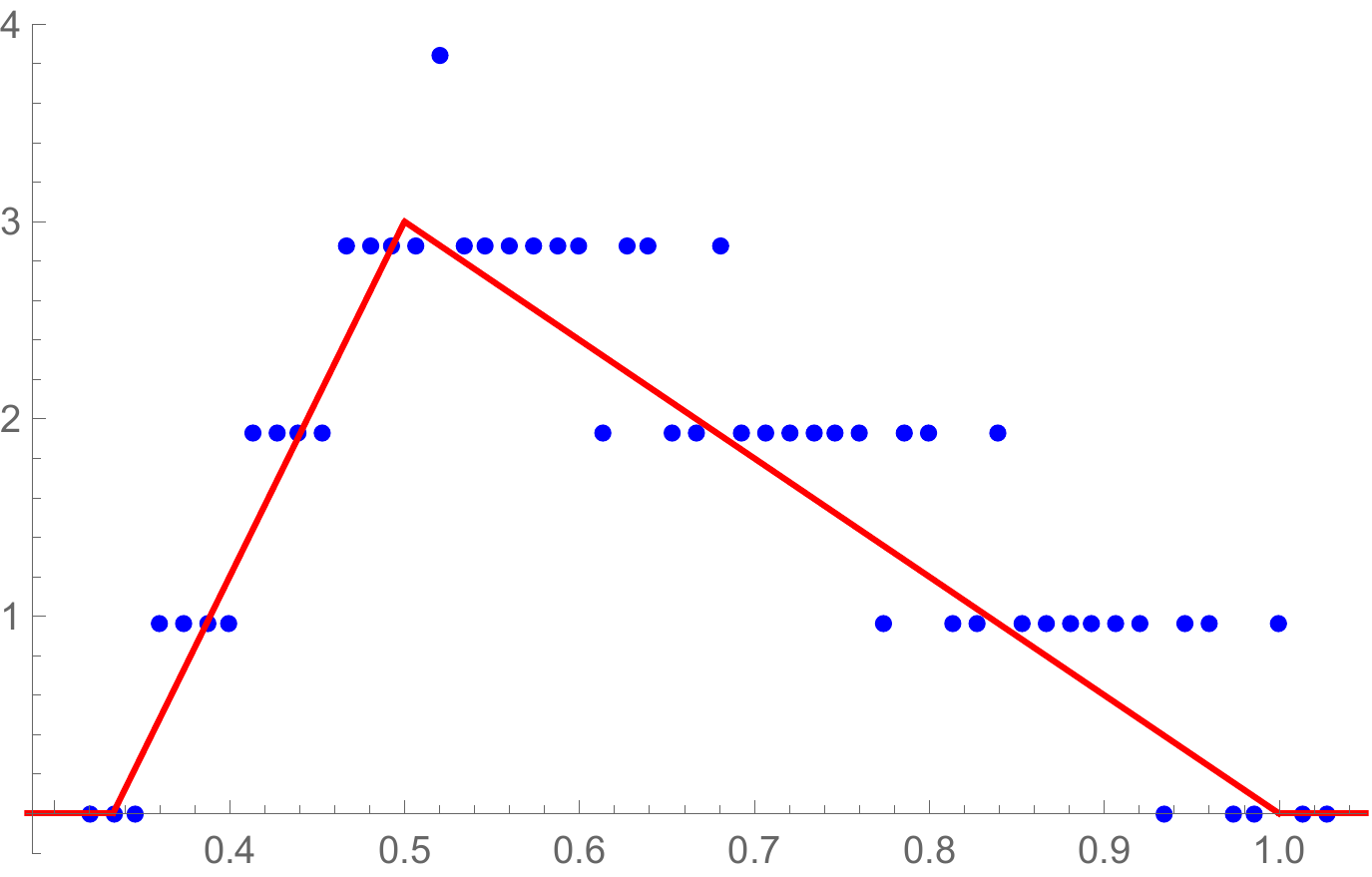} 
	                \caption{$\vec{m}_2=(3,1,2)$ and $\vec{n}_2=(3,2,6)$}
	        \end{subfigure}
\caption{Different values of $\vec{m}$ and $\vec{n}$ can produce scaled histograms that are translations of each other.
In the context of Proposition~\ref{Proposition:supersymmetric}, $(a,b,c)=(1,2,3)$.}
\label{Figure:supersymmetric}
\end{figure}

Proposition~\ref{Proposition:supersymmetric} says that two different weighted lengths on the same numerical semigroup yield nearly the same (translated) statistical behavior.  This is consistent with Theorem~\ref{Theorem:Main}
since
\begin{equation*}
\left\{\frac{a c}{a},\frac{a b}{b}, \frac{b c}{c}\right\}=\left\{a,b,c\right\}= \left\{\frac{a c}{c},\frac{a b}{a}, \frac{b c}{b}\right\},
\end{equation*}
so the asymptotic distribution functions in the two cases are equal.  
The numerical semigroups in Proposition~\ref{Proposition:supersymmetric} are called \emph{supersymmetric}~\cite{bryant2013maximal}.

\begin{proposition}\label{Proposition:supersymmetric}
Let $a,b,c\in \Z_{>0}$ be distinct, and let $\vec{m}_1=(b, a, c)$, $\vec{m}_2=(a, c, b)$, and $\vec{n}=(a b, a c, b c)$. Define
\begin{equation*}
A_1=\begin{bmatrix}\vec{m}_1^{\T} \\[3pt] \vec{n}^{\T} \end{bmatrix}
\in \M_{2\times 3}(\Z)
\qquad \text{and} \qquad
A_2=\begin{bmatrix}\vec{m}_2^{\T} \\[3pt] \vec{n}^{\T} \end{bmatrix}
\in \M_{2\times 3}(\Z),
\end{equation*}
along with
\begin{equation*}
\ZZ_1(m,n)=\left\{\vec{x}\in \Z_{\geq 0}^3:A_1\vec{x}=\begin{bmatrix}m\\n\end{bmatrix}\right\}
\quad \text{and} \quad
\ZZ_2(m,n)=\left\{\vec{x}\in \Z_{\geq 0}^3:A_2\vec{x}=\begin{bmatrix}m\\n\end{bmatrix}\right\}.
\end{equation*}
For all $n \in \Z_{\geq 0}$, there exists an $r_n \in \Z$ such that for all $m \in \Z$
\begin{equation*}
|\ZZ_1(m,n)|=|\ZZ_2(m+r_n,n)|.
\end{equation*}
Moreover, $r_n=r_{n+a b c}$ for all $n \in \Z_{\geq 0}$.
\end{proposition}

\begin{proof}
Let $S = \<ab, ac, bc\>$ and fix $n \in S$.  We can write $n = qabc + r$ with $r \in S$ and $r - abc \notin S$.  By~\cite[Prop.~1, Thm.~12]{garcia2013affine}, we have $|\mathsf Z_S(r)| = 1$,
\begin{equation*}
\mathsf Z_S(qabc) = \{(z_1c, z_2b, z_3a) \in \Z_{\geq 0}^3 : z_1 + z_2 + z_3 = q\},
\end{equation*}
and
\begin{equation*}
\mathsf Z_S(n) = \mathsf Z_S(qabc) + \mathsf Z_S(r).
\end{equation*}
For any $z_1, z_2, z_3 \in \Z_{\geq 0}$ with $z_1 + z_2 + z_3 = q$, we have
\begin{equation*}
(z_1c, z_2b, z_3a) \cdot \vec m_1 = z_1bc + z_2ab + z_3ac = (z_3c, z_1b, z_2a) \cdot \vec m_2,
\end{equation*}
which implies $|\ZZ_1(m, qabc)|=|\ZZ_2(m, qabc)|$ for all $m \in \Z$.  Writing $\mathsf Z_S(r) = \{\vec x\}$, linearity then implies 
\begin{equation*}
\ZZ_1(m + (\vec m_1 \cdot \vec x), n) = \ZZ_1(m, qabc) + \vec x
\,\, \text{and} \,\,
\ZZ_2(m + (\vec m_2 \cdot \vec x), n) = \ZZ_2(m, qabc) + \vec x
\end{equation*}
for all $m \in \Z$.  This yields the desired claim upon letting $r_n = (\vec m_2 - \vec m_1) \cdot \vec x$.  
\end{proof}

%%%%%%%%%%%%%%%%%%%%%%%%%%%%%%%%%%%%%%%%%%%%%%%%%%%%%%%%

\section{Proof of Theorem \ref{Theorem:Main}}\label{Section:Proof}
The proof of Theorem~\ref{Theorem:Main} is geometric:
the limiting distribution arises from the projection of a simplex with one vertex on each axis,
with each vertical value in the distribution being the volume of a cross section.  
This yields a piecewise-polynomial function; the transition between each 
polynomial piece occurs when the cross section contains a vertex.
Making this general and precise, with explicit error bounds, adds to the complexity of the argument.

%%%%%%%%%%%%%%%%%%%%%%%%%%%%%%%%%%%%%%%%
\subsection{Setup}
Let $\vec{m}=(m_1,m_2,m_3)$, $\vec{n}=(n_1,n_2,n_3)$, and 
\begin{equation*}
A =
\begin{bmatrix}
m_1 & m_2 & m_3 \\
n_1 & n_2 & n_3 \\
\end{bmatrix}
= \twovector{\vec{m}^{\T}}{ \vec{n}^{\T}} 
\in \M_{2\times 3}(\Z).
\end{equation*}
The hypotheses on the ratios $m_i/n_i$ imply that
\begin{equation}\label{eq:Rho}
 \underbrace{m_2 n_3 - m_3 n_2}_{\rho_1} \geq 0,
\qquad
 \underbrace{m_1 n_3 - m_3 n_1}_{\rho_2} > 0,
\quad\text{and}\quad
\underbrace{m_1 n_2 - m_2 n_1}_{\rho_3} \geq 0.
\end{equation}
Observe that
\begin{equation*}
\rho_1 = 0 \quad \iff \quad \frac{m_3}{n_3}=\frac{m_2}{n_2}
\qquad \text{and} \qquad
\rho_3 = 0 \quad \iff \quad \frac{m_2}{n_2}=\frac{m_1}{n_1},
\end{equation*}
so at most one of $\rho_1,\rho_3$ can be zero, since otherwise $\rho_2= 0$ and $m_3/n_3=m_2/n_2=m_1/n_1$. 
Treat the corresponding interval $[\frac{m_3}{n_3},\frac{m_2}{n_2}]$ or $[\frac{m_2}{n_2},\frac{m_1}{n_1}]$ 
 as degenerate in these cases.  This also means that at least two of the three inequalities in \eqref{eq:Rho} are strict.

The one-dimensional subspace $\{\vec{m}\}^{\perp} \cap \{\vec{n}\}^{\perp}$ of $\R^3$ is spanned by 
\begin{equation*}
\vec{r}= \vec{m} \times \vec{n} = 
\begin{bmatrix}
m_2 n_3 - m_3 n_2 \\
m_3 n_1 - m_1 n_3   \\
m_1 n_2 - m_2 n_1 
\end{bmatrix}
= 
\begin{bmatrix}
\rho_1 \\  -\rho_2 \\ \rho_3
\end{bmatrix}
\in \Z^3 \backslash\{ \vec{0}\}.
\end{equation*}
By construction, $A \vec{r} = \vec{0}$.
Define $\lambda(\vec{x})$ and $\Lambda\multi{n}$ as in the statement of Theorem \ref{Theorem:Main} and note that $\lambda(\vec{x})=\vec{m}\cdot \vec{x}$. 

\subsection{The sets $\ZZ(\vec{b})$ and $\widetilde{\ZZ}(\vec{b})$}
We adjust the notation \eqref{eq:ZZb} to permit vector arguments: for $\vec{b} =(m,n) \in \Z^2$, let
\begin{equation}
\ZZ(\vec{b}) 
= \{\vec{x}\in \Z_{\geq 0}^3: A \vec{x} = \vec{b} \} 
= \{ \vec{x} \in \ZZZ_S(n) : \lambda(\vec{x})=m\}.
\end{equation}
Similarly, define
\begin{equation*}
    \widetilde{\ZZ}(\vec{b})=\{\vec{x}\in \Z^3: A \vec{x} = \vec{b} \}.
\end{equation*}
We may denote these as $\ZZ(m,n)$ and $\widetilde{\ZZ}(m,n)$, respectively, as convenient.
Both $\ZZ(\vec{b})$ and $\widetilde{\ZZ}(\vec{b})$ may be empty; the following lemma 
gives some crucial insight on when $\widetilde{\ZZ}(\vec{b})$ is empty.
Although the lemma is a special case of \cite[Thm.~3.2b]{MR4142067}, we provide 
another proof since the three-dimensional setting permits the use of the cross product and geometric reasoning to simplify the argument.

\begin{lemma}\label{Lemma:whenEmpty}
Let $d=\gcd(\rho_1,\rho_2,\rho_3)$. 
For each $n \in \Z$, there is some $c \in \{0,1,2,\ldots,d-1\}$ such that $\widetilde{\ZZ}(m,n)\neq\varnothing$ 
if and only if $m\equiv c \pmod{d}$.
\end{lemma}

\begin{proof}
The definition of $d$ ensures that, $\rho_i \equiv 0 \pmod{d}$ for $i=1,2,3$.  Thus,
\begin{equation*}
m_in_j \equiv m_jn_i \pmod{d}
\end{equation*}
for $i,j=1,2,3$.
For any $\vec{x}=(x_1,x_2,x_3)\in \Z^3$ and $i\in \{1,2,3\}$, it follows that
\begin{equation*}
m_i(n_1x_1+n_2x_2+n_3x_3) \equiv n_i(m_1x_1+m_2x_2+m_3x_3) \pmod{d};
\end{equation*}
that is,
\begin{equation*}
m_i(\vec{n}\cdot \vec{x}) \equiv n_i(\vec{m}\cdot \vec{x}) \pmod{d}.
\end{equation*}

\medskip\noindent$(\Rightarrow)$
Suppose there is an $\vec{x}\in \widetilde{\ZZ}(m,n)$.
Since $\gcd(n_1,n_2,n_3)=1$, B\'ezout's identity provides
$a_1,a_2,a_3\in \Z$ such that $a_1n_1+a_2n_2+a_3n_3 =1$. 
Let $c$ denote the least nonnegative residue of $(a_1m_1+a_2m_2+a_3m_3)n$ modulo $d$.  
Then
\begin{align*}
m 
&= \vec{m}\cdot \vec{x} \\
&=(a_1n_1+a_2n_2+a_3n_3)(\vec{m}\cdot \vec{x})\\
&\equiv (a_1m_1+a_2m_2+a_3m_3)(\vec{n}\cdot \vec{x}) \pmod{d}\\
&\equiv (a_1m_1+a_2m_2+a_3m_3) n \pmod{d}\\
&\equiv c \pmod{d}.
\end{align*}

\noindent$(\Leftarrow)$
Since $d=\gcd(\rho_1,\rho_2,\rho_3)$,
B\'ezout's identity provides a $\vec{v}\in \Z^3$ such that
\begin{equation*}
\vec{r}\cdot \vec{v} = (\rho_1, -\rho_2, \rho_3)\cdot \vec{v} = d.
\end{equation*}
Let $\vec{w}=\vec{n}\times \vec{v}$ and observe that
\begin{equation*}
\vec{n}\cdot \vec{w}= \vec{n}\cdot (\vec{n}\times \vec{v}) = (\vec{n}\times \vec{n})\cdot \vec{v} = 0
\end{equation*}
and
\begin{equation*}
\vec{m} \cdot \vec{w} = \vec{m} \cdot (\vec{n} \times \vec{v}) = (\vec{m}\times \vec{n}) \cdot \vec{v}=\vec{r} \cdot \vec{v} = d.
\end{equation*}
Fix $n \in \Z_{\geq 0}$. 
Since $\gcd(n_1,n_2,n_3)=1$, there is a $\vec{z}=(z_1,z_2,z_3)\in \Z^3$ such that $\vec{n}\cdot \vec{z}=n$. 
Let $s=\vec{m}\cdot \vec{z}$, so that $\vec{z} \in \widetilde{\ZZ}(s,n)$.
The first half of the proof ensures that $s\equiv c \pmod{d}$. 
If $m\equiv c \pmod{d}$, then $d\mid (m-s)$ and hence
\begin{equation*}
\vec{m} \cdot \left(\vec{z}+\frac{m-s}{d}\vec{w}\right) 
= \vec{m} \cdot \vec{z} + \frac{m-s}{d} \vec{m} \cdot \vec{w}
=s + \frac{m-s}{d}d = m.
\end{equation*}
Therefore,
\begin{equation*}
\vec{z}+\frac{m-s}{d}\vec{w} \in \widetilde{\ZZ}\left(s+\frac{m-s}{d}d,n\right)=\widetilde{\ZZ}(m,n).\qedhere
\end{equation*}
\end{proof}

\begin{lemma}\label{Lemma:GetZed}
Let $\vec{b} \in \Z^2$.  If $\vec{z}\in \widetilde{\ZZ}(\vec{b})$, then $\widetilde{\ZZ}(\vec{b})=\{\vec{z}+s\vec{r}/d:s\in \Z\}$ where $d=\gcd(\rho_1,\rho_2,\rho_3)$.
\end{lemma}

\begin{proof}
Since $A \vec{r} = \vec{0}$, we have $A(\vec{z} + s \vec{r}/d) = \vec{b}$.  Additionally, $\vec{r}/d\in \Z^3$ because $d=\gcd(\rho_1,\rho_2,\rho_3)$. Therefore, $\{\vec{z}+s\vec{r}/d : s\in \Z\} \subseteq \widetilde{\ZZ}(\vec{b})$. 
Suppose that  $\vec{x}\in \widetilde{\ZZ}(\vec{b})$. Then
$A(\vec{x} - \vec{z}) = A \vec{x} - A \vec{z} = \vec{0}$, so
$\vec{x} - \vec{z} = s\vec{r}/d$ for some $s \in \R$.  
Then $s \vec{r}/d = \vec{x}-\vec{z}\in \Z^3$, and hence $s \in \Z$ because $\gcd(\rho_1,\rho_2,\rho_3) = d$.
Thus, $\widetilde{\ZZ}(\vec{b}) \subseteq \{\vec{z}+s\vec{r}/d:s\in \Z\}$.
\end{proof}

%%%%%%%%%%%%%%%%%%%%%%%%%%%%%%%%%%%%%%%%
\subsection{Some geometry}
For $\vec{y} =(y_1,y_2) \in \R^2$, let $\ell(\vec{y})$ denote the length of the line segment 
\begin{equation*}
\L(\vec{y}) = \{\vec{x}\in \R_{\geq 0}^3 :  A \vec{x} = \vec{y} \}
\end{equation*}
if it is nonempty; let $\ell(\vec{y}) = 0$ otherwise. 
On occasion, we may write $\L(y_1,y_2)$ and $\ell(y_1,y_2)$ instead.
 The line $\L(\vec{y})$ is contained in the plane $\{\vec{x} \in \R^3:\vec{n}\cdot \vec{x}=y_2\}$ which, owing to the positivity of the components of $\vec{n}$, 
 has compact intersection with $\R_{\geq 0}^3$. Thus, $\ell(\vec{y})$ is finite.
Observe that for $n \in \Z_{n > 0}$,
\begin{equation}\label{eq:Scaling}
n \ell( \tfrac{x}{n},1) = \ell(x,n).
\end{equation}

\begin{lemma}\label{Lemma:LineCount}
Let $d=\gcd(\rho_1,\rho_2,\rho_3)$ and $\vec{b}\in \Z^2$.
If $\widetilde{\ZZ}(\vec{b})\neq \varnothing$, then
\begin{equation*}
\frac{d \ell(\vec{b})}{\norm{\vec{r}}}-1
\,\leq\, |\ZZ(\vec{b})|
\,\leq\, \frac{d \ell(\vec{b})}{\norm{\vec{r}}}+1 .
\end{equation*}
\end{lemma}

\begin{proof}
Suppose that $\ZZ(\vec{b}) \neq \varnothing$.
Then Lemma \ref{Lemma:GetZed} provides a $\vec{z}\in \Z^3$ such that
\begin{equation*}
\ZZ(\vec{b})=\Z_{\geq 0}^3\cap \{\vec{z}+s\vec{r}/d : s\in \Z\}.
\end{equation*}
Define
\begin{equation*}
a = \inf \{s \in \R:\vec{z}+s\vec{r}/d \in \R_{\geq 0}^3\}
\qquad \text{and} \qquad
b = \sup \{s \in \R:\vec{z}+s\vec{r}/d \in \R_{\geq 0}^3\}.
\end{equation*}
Then $\vec{z}+s\vec{r}/d \in \R_{\geq 0}^3$ if and only if $s\in [a,b]$.
Consequently,
\begin{equation*}
|\ZZ(\vec{b})|
= | [a,b]  \cap \Z| 
= \lfloor b\rfloor - \lceil a \rceil +1 .
\end{equation*}
Since
\begin{equation*}
b-1<\floor*{b} \leq b 
\qquad \text{and} \qquad
-a-1\leq -\ceil*{a}\leq -a,
\end{equation*}
it follows that
\begin{equation}\label{eq:PlugHere}
b-a-1\leq | \ZZ(\vec{b}) | \leq b-a+1. 
\end{equation}
The length of $\L(\vec{b})$ is
\begin{equation*}
\ell(\vec{b})=\norm{(\vec{z}+b\vec{r}/d)-(\vec{z}+a\vec{r}/d)}=\frac{b-a}{d}\norm{\vec{r}}.
\end{equation*} 
Substitute $b-a =d\ell(\vec{b})/ \norm{ \vec{r}}$ in \eqref{eq:PlugHere} and obtain the desired inequalities.
\end{proof}

%%%%%%%%%%%%%%%%%%%%%%%%%%%%%%%%%%%%%%%%
\subsection{The triangle emerges}
Recall that $f:I\to\R$ is \emph{Lipschitz} on a (possibly infinite) interval $I$ with
\emph{Lipschitz constant} $C$ if $|f(x) - f(y) | \leq C |x-y|$ for all $x,y \in I$.

\begin{lemma}\label{Lemma:Triangle}
Suppose that $\rho_1, \rho_3 \neq 0$. 
For $t \in \R$,
\begin{equation*}
\ell(t,1)
= \frac{\norm{ \vec{r} }}{\rho_2}
\begin{cases}
0 & \text{if $t < \frac{m_3}{n_3}$},\\[5pt]
\dfrac{n_3 t - m_3 }{  \rho_1 } & \text{if $\frac{m_3}{n_3} \leq t \leq \frac{m_2}{n_2}$},\\[8pt]
 \dfrac{m_1 - n_1 t }{ \rho_3  } & \text{if $\frac{m_2}{n_2} \leq t \leq \frac{m_1}{n_1}$},\\[5pt]
0 & \text{if $t > \frac{m_1}{n_1}$}.
\end{cases}
\end{equation*}
is a ``triangular'' function of $t$ with base $[ \frac{m_3}{n_3}, \frac{m_1}{n_1}]$, peak
at $t = \frac{ m_2}{n_2}$, and height
\begin{equation*}
\ell\Big(\frac{ m_2}{n_2},1\Big) = \frac{\norm{ \vec{r}} }{n_2  \rho_2}.
\end{equation*}
Furthermore, $\ell(t,1)$ is Lipschitz with Lipschitz constant
\begin{equation*}
\frac{\norm{ \vec{r}}}{\rho_2} \operatorname{max}\left\{ \frac{n_3}{\rho_1}, \frac{n_1}{\rho_3} \right\} .
\end{equation*}
\end{lemma}

\begin{proof}
If it is nonempty, the line segment $\L(t,1)$ lies in $\R^3_{\geq 0}$; its endpoints each lie on one of the coordinate planes.  
Solve the corresponding equations and obtain the points of intersection with the three coordinate planes:
\begin{itemize}[leftmargin=*]
\item $\vec{p}_1(t) = \rho_1^{-1} (0, \,\, n_3 t-m_3, \, \, m_2-n_2 t)$, hence $\vec{p}_1(t) \in \R_{\geq 0}^3 \iff t \in [ \frac{m_3}{n_3}, \frac{m_2}{n_2}]$,
\item $\vec{p}_2(t) = \rho_2^{-1} (n_3 t-m_3,\,\,  0,\,\,  m_1-n_1 t)$, hence $\vec{p}_2(t) \in \R_{\geq 0}^3 \iff t \in [ \frac{m_3}{n_3}, \frac{m_1}{n_1}]$,
\item $\vec{p}_3(t) = \rho_3^{-1} (n_2 t-m_2,\,\,  m_1-n_1 t,\,\, 0)$, hence $\vec{p}_3(t) \in \R_{\geq 0}^3 \iff t \in [ \frac{m_2}{n_2}, \frac{m_1}{n_1}]$,
\end{itemize}
since $\rho_1,\rho_2,\rho_3 \geq 0$.
In particular, if $\rho_1=0$ or $\rho_3=0$, then $\L(t,1)$ does not meet the corresponding coordinate plane in $\R_{\geq 0}^3$
(recall that $\rho_2 > 0$).  

For $t < \frac{m_3}{n_3}$ or $t > \frac{m_1}{n_1}$, we have $\ell(t,1) = 0$.
For $t \in [\frac{m_3}{n_3}, \frac{m_2}{n_2}]$, we see that $\L(t,1)$ is the line segment from
$\vec{p}_1(t)$ to $\vec{p}_2(t)$.  A computation confirms that
\begin{align*}
\ell(t,1)
&= \norm{ \vec{p}_1(t) - \vec{p}_2(t) }
= \frac{n_3  t - m_3 }{ \rho_1 \rho_2 }\norm{ \vec{r}}.
\end{align*}
For $t \in [\frac{m_2}{n_2}, \frac{m_1}{n_1}]$, we see that $\L(t,1)$ is the line segment from
$\vec{p}_2(t)$ to $\vec{p}_3(t)$, so
\begin{align*}
\ell(t,1)
&= \norm{ \vec{p}_2(t) - \vec{p}_3(t) }
= \frac{m_1 - n_1 t }{ \rho_2 \rho_3 }\norm{ \vec{r}}
\end{align*}
via another computation.
This yields the desired piecewise-linear formula for $\ell(t,1)$.
An admissible Lipschitz constant is the maximum of the slopes of $\ell(t,1)$
on $[\frac{m_3}{n_3}, \frac{m_2}{n_2}]$ and $[\frac{m_2}{n_2}, \frac{m_1}{n_1}]$,
so long as the corresponding interval is nondegenerate.
Elementary computations confirm the remainder of the lemma.
\end{proof}

\begin{remark}\label{Remark:Lipschitz}
If $\rho_1 = 0$ or $\rho_3 = 0$ (the conditions are mutually exclusive), then the corresponding interval in the definition of $\ell(t,1)$
and term in the maximum above are omitted. Moreover, $\ell(t,1)$ is Lipschitz on 
$[ \frac{m_3}{n_3}, \infty)$ or $(-\infty, \frac{m_1}{n_1}]$, respectively.
\end{remark}

Lemma \ref{Lemma:Triangle} states that $\ell(x,1)$ is a triangular function with base 
$[\frac{m_3}{n_3}, \frac{m_1}{n_1}]$ and height $\frac{\norm{\vec r}}{n_2\rho_2}$. 
Since the base width is
\begin{equation*}
\frac{m_1}{n_1}-\frac{m_3}{n_3}=\frac{m_1n_3-n_1m_3}{n_1n_3}=\frac{\rho_2}{n_1n_3},
\end{equation*}
the area of the triangle is
\begin{equation*}
\int_{\R}\ell(x,1)\,dx 
= \frac{1}{2}\cdot\frac{\rho_2}{n_1n_3}\cdot \frac{\norm{\vec r}}{n_2 \rho_2}
=\frac{\norm{\vec{r}} }{2n_1n_2n_3}.
\end{equation*}
In particular,
\begin{equation}\label{eq:Fell}
F(t) = 2 n_1 n_2 n_3 \frac{\ell(t,1)}{\norm{ \vec{r}}}
\end{equation}
is the probability density from Theorem \ref{Theorem:Main}.

%%%%%%%%%%%%%%%%%%%%%%%%%%%%%%%%%%%%%%%%
\subsection{A technical lemma}
The next lemma permits us to approximate a discrete sum by an integral 
with a completely explicit error estimate.  

\begin{lemma}\label{Lemma:Convergence}
Suppose that
\begin{enumerate}
\item $g:\R\to\R$ satisfies $|g(x)|\leq C_1$ for $x \in \R$;
\item $g$ is Lipschitz on some closed interval $I$ with Lipschitz constant $C_2$;
\item $n ,c, d\in \Z_{\geq 0}$ and $c<d$;
\item $f:\Z\to \Z$ satisfies $$\displaystyle \bigg|f(c+kd)/d- n g\bigg(\frac{c+kd}{n} \bigg)\bigg| \leq 1$$ for $k\in \Z$; and
\item $f(x)=0$ for $x\not \equiv c \pmod{d}$.
\end{enumerate}
Then for real $\alpha < \beta$ such that $[\alpha, \beta]\subseteq I$,
\begin{equation*}
\left| \frac{1}{n^2}\sum_{k \in \Z\cap [\alpha n, \beta n]}f(k )  - \int_{\alpha}^{\beta} g(x)\,dx \right|
\leq \frac{(\beta-\alpha +\frac{2d}{n})(1+dC_2)+d(5C_1+\frac{1}{n})}{n}.
\end{equation*}
\end{lemma}

\begin{proof}
Since the proof is somewhat long, we break it up into several pieces.

\medskip\noindent\textbf{An auxiliary function.}
Let $G(x)=\frac{n}{d}f(c+\floor*{nx/d} d)$. Then
\begin{align}
\int_{\frac{d}{n}\ceil{ \frac{\alpha n-c}{d} }}^{\frac{d}{n}\floor{ \frac{\beta n-c}{d}+1}}G(x)\,dx 
&=\int_{\ceil{ \frac{\alpha n-c}{d} }}^{\floor{ \frac{\beta n-c}{d}+1}}\frac{d}{n}G\Big(\frac{d}{n} u\Big)\,du \nonumber\\
&=\int_{\ceil{ \frac{\alpha n-c}{d} }}^{\floor{ \frac{\beta n-c}{d}+1}} f(c+\floor{u} d) \, du \nonumber\\
&= \sum_{k \in \Z\cap [\frac{\alpha n-c}{d}, \frac{\beta n-c}{d}]}f(c+ kd) \nonumber\\
&=\sum_{k\in \Z\cap [\alpha n,\beta n]} f(k) .\label{eq:G2f}
\end{align}
For $k\in \Z$, condition (d) ensures that
\begin{align}
    \abs{\frac{G(kd /n)}{n^{2}}-g\Big(\frac{c+kd}{n}\Big)}
    &=\abs{\frac{\frac{n}{d}f\big(c+\big\lfloor n\frac{k d}{nd} \big\rfloor d\big)}{n^{2}}- g\left(\frac{c+kd}{n}\right)} \nonumber\\
    &=\abs{\frac{f(c+kd )/d - n g(\frac{c+kd}{n} )}{n}} \nonumber\\
    &\leq \frac{1}{n}. \label{eq:Final1n}
\end{align}
We also need a bound afforded by (a) and (d):
\begin{equation}\label{eq:gC1}
\bigg|\frac{G(x)}{n^2} \bigg|= \frac{ | f(c+\floor{nx/d} d) | }{dn}\leq \frac{1}{n}+\left|g \bigg(\frac{c+\floor{nx/d} d}{n} \bigg)\right| \leq 
C_1+\frac{1}{n}.
\end{equation}

\medskip\noindent\textbf{A Lipschitz estimate.}
Observe that $G(x)=G(\floor{nx/d}d/n)$ and
\begin{equation*}
\frac{-d}{n} \,\leq\, \frac{c+nx-d}{n}-x \,\leq\, \frac{c+\floor*{nx/d}d}{n}-x  \,\leq\, \frac{c+nx}{n}-x \,\leq\, \frac{d}{n}.
\end{equation*}
If $x$ and $(c+\floor*{nx/d}d)/n$ are both in $I$, condition (b) and \eqref{eq:Final1n} imply that
\begin{align}
     \left|\frac{G(x)}{n^2}-g(x) \right|
     &\leq \left| \frac{G(\floor*{nx/d} d/n)}{n^2} -g\left(\frac{c+\floor*{nx/d} d}{n}\right) \right| \nonumber\\
     &\qquad\qquad +\left| g\left(\frac{c+\floor*{nx/d} d}{n}\right)-g(x)\right| \nonumber \\
     &\leq \frac{1+dC_2}{n}.\label{eq:FirstG}
\end{align}

\medskip\noindent\textbf{A containment.} 
We claim $x$ and $(c+\floor{nx/d}d)/n$ belong to $[\alpha , \beta] \subset I$ whenever
\begin{equation*}
\frac{d}{n}\ceil*{\frac{\alpha n-c}{d}+1} \, \leq \,  x \, \leq \,  \frac{d}{n}\floor*{\frac{\beta n-c}{d}}.
\end{equation*}
Suppose that the inequality above holds.  Then
\begin{equation*}
x\geq \frac{d}{n}\left(\frac{\alpha n -c}{d}+1\right)=\alpha+\frac{d-c}{n}  \geq \alpha
\end{equation*}
and
\begin{equation*}
x\leq \frac{d}{n}\floor*{\frac{\beta n-c}{d}} \leq \frac{d}{n}\cdot \frac{\beta n-c}{d} = \beta - \frac{c}{n} \leq  \beta.
\end{equation*}
Next observe that
\begin{equation*}
\frac{c+\floor{\frac{n}{d}x}d}{n}\geq \frac{c+\floor{\frac{n}{d} ( \frac{d}{n}\ceil{\frac{\alpha n-c}{d}+1} ) }d}{n} 
= \frac{c+\ceil{\frac{\alpha n-c}{d}+1}d}{n} \geq  \frac{c+\alpha n -c+d}{n} \geq  \alpha
\end{equation*}
and
\begin{equation*}
\frac{c+\floor{\frac{n}{d}x}d}{n} \leq \frac{c+\floor{\frac{n}{d}  (  \frac{d}{n}\floor{\frac{\beta n-c}{d}}  )  }d}{n} 
= \frac{c+\floor{\frac{\beta n-c}{d}}d}{n} \leq \frac{c+\beta n -c}{n}=\beta.
\end{equation*}
This completes the proof of the claim.

\medskip\noindent\textbf{An observation.}
Since
\begin{equation*}
\frac{d}{n}\floor*{\frac{\beta n -c}{d}}-\frac{d}{n}\ceil*{\frac{\alpha n - c}{d}+1}\leq \beta -\frac{c}{n}-\left(\alpha+\frac{d-c}{n}\right)=\beta-\alpha-\frac{d}{n}
\end{equation*}
and
\begin{equation*}
\frac{d}{n}\ceil*{\frac{\alpha n - c}{d}+1}-\frac{d}{n}\floor*{\frac{\beta n - c}{d}}<\alpha +\frac{2d-c}{n}-\left(\beta-\frac{c}{n}\right)=\alpha-\beta +\frac{2d}{n},
\end{equation*}
we conclude that
\begin{equation}\label{eq:MaxLength}
\Bigg| \frac{d}{n}\floor*{\frac{\beta n-c}{d}} -  \frac{d}{n}\ceil*{\frac{\alpha n-c}{d}+1} \Bigg|
\leq \operatorname{max}\left\{\beta-\alpha-\frac{d}{n} , \, \alpha-\beta+\frac{2d}{n}\right\}.
\end{equation}

\medskip\noindent\textbf{Small intervals.} 
Consider the intervals $\big[\alpha, \frac{d}{n} \ceil{ \frac{\alpha n - c}{d} + 1}\big]$ and
$\big[ \frac{d}{n} \floor{ \frac{\beta n - c}{d} } , \beta \big]$.
Since
\begin{align*}
0 &<\frac{d-c}{n} = \frac{d}{n}\left(\frac{\alpha n-c}{d}+1\right) - \alpha \\
&\leq \frac{d}{n}\ceil*{\frac{\alpha n - c}{d}+1} - \alpha 
< \frac{d}{n}\left(\frac{\alpha n -c}{d}+2\right)-\alpha \\
&=\frac{2d-c}{n}, 
\end{align*}
the first interval is nonempty with length at most $(2d-c)/n$.
Similarly,
\begin{align*}
0
&\leq \frac{c}{n} = \beta - \frac{d}{n} \left(\frac{\beta n - c}{d} \right)
\leq \beta - \frac{d}{n} \floor*{ \frac{\beta n - c}{d} } \\
&<\beta - \frac{d}{n}\left(  \frac{\beta n - c}{d}-1 \right)  
= \frac{c+d}{n},
\end{align*}
so the second interval has length at most $(c+d)/n$.
In summary,
\begin{equation}\label{eq:cndn}
0 < \frac{d}{n}\ceil*{\frac{\alpha n -c}{d}+1}-\alpha < \frac{2d-c}{n}
\quad \text{and} \quad
0 \leq  \beta -\frac{d}{n}\floor*{\frac{\beta n -c}{d}} < \frac{c+d}{n}.
\end{equation}

\medskip\noindent\textbf{Conclusion.}
We conclude that
\begin{align*}
    &\Bigg| \frac{1}{n^2} \sum_{k \in \Z\cap [\alpha n, \beta n]}f(k )  - \int_{\alpha}^{\beta} g(x)\,dx \Bigg|  \\
    &=\left| \int_{ \frac{d}{n}\ceil{\frac{\alpha n-c}{d}}}^{ \frac{d}{n}\floor{\frac{\beta n-c}{d}+1}} \frac{G(x)}{n^{2}}\,dx - \int_{\alpha}^{\beta} g(x)\,dx \right| && \text{by \eqref{eq:G2f}}\\
    &=\Bigg| 
     \int_{\frac{d}{n}\ceil{\frac{\alpha n-c}{d}}}^{\frac{d}{n}\ceil{\frac{\alpha n-c}{d}+1}}\frac{G(x)}{n^2}\,dx 
     + \int_{ \frac{d}{n}\ceil{\frac{\alpha n-c}{d}+1}}^{ \frac{d}{n}\floor{\frac{\beta n-c}{d}}} \frac{G(x)}{n^{2}} \,dx \\
     &\hspace{1in}
     +\int_{\frac{d}{n}\floor{\frac{\beta n-c}{d}}}^{\frac{d}{n}\floor{\frac{\beta n-c}{d}+1}}\frac{G(x)}{n^2}\,dx
    -\int_{\alpha}^{\beta}g(x)\,dx \Bigg|\\
    &\leq
    \int_{\frac{d}{n}\ceil{\frac{\alpha n-c}{d}}}^{\frac{d}{n}\ceil{\frac{\alpha n-c}{d}+1}}\left|\frac{G(x)}{n^2}\right|\,dx
    +\int_{\frac{d}{n}\floor{\frac{\beta n-c}{d}}}^{\frac{d}{n}\floor{\frac{\beta n-c}{d}+1}}\left|\frac{G(x)}{n^2}\right|\,dx \\
    &\hspace{1in}
    +    \left| \int_{ \frac{d}{n}\ceil{\frac{\alpha n-c}{d}+1}}^{ \frac{d}{n}\floor{\frac{\beta n-c}{d}}} \frac{G(x)}{n^{2}} \,dx 
    -\int_{\alpha}^{\beta}g(x)\,dx \right|   \\
    &\leq \bigg| \int_{ \frac{d}{n}\ceil{\frac{\alpha n-c}{d}+1} }^{ \frac{d}{n}\floor{\frac{\beta n-c}{d}}} \left(\frac{G(x)}{n^{2}}-g(x) \right)\,dx  
    -\int_{\alpha}^{\frac{d}{n}\ceil{\frac{\alpha n-c}{d}+1} }g(x)\,dx  \\
    &\hspace{1in}-\int_{\frac{d}{n}\floor{\frac{\beta n-c}{d}} }^{\beta}g(x)\,dx \bigg| +\frac{2d}{n}\left(C_1+\frac{1}{n}\right) && \text{by \eqref{eq:gC1}}\\
    &\leq \left|\int_{ \frac{d}{n}\ceil{\frac{\alpha n-c}{d}+1} }^{ \frac{d}{n}\floor{\frac{\beta n-c}{d}}} \left|\frac{G(x)}{n^{2}}-g(x) \right|\,dx\right| 
      +\int_{\alpha}^{\frac{d}{n}\ceil{\frac{\alpha n-c}{d}+1} }|g(x)|\,dx \\
    &\hspace{1in}+\int_{\frac{d}{n}\floor{\frac{\beta n-c}{d}} }^{\beta}|g(x)|\,dx 
       +\frac{2d}{n}\left(C_1+\frac{1}{n}\right)\\
    &\leq 
      \Bigg| \frac{d}{n}\floor*{\frac{\beta n-c}{d}} -  \frac{d}{n}\ceil*{\frac{\alpha n-c}{d}+1} \Bigg| \bigg(\frac{1+dC_2}{n}\bigg) 
            && \text{by \eqref{eq:FirstG}}\\
     &\hspace{0.5in} +\frac{(2d-c)C_1}{n}
     +\frac{(c+d)C_1}{n} 
     +\frac{2d}{n}\left(C_1+\frac{1}{n}\right)  && \text{by (a), \eqref{eq:cndn}} \\
     &\leq \operatorname{max}\left\{\beta-\alpha-\frac{d}{n} , \, \alpha-\beta+\frac{2d}{n}\right\}
     \left(\frac{1+dC_2}{n}\right)      && \text{by \eqref{eq:MaxLength}}\\
 &\hspace{1in}+\frac{3dC_1}{n} +\frac{2d}{n}\left(C_1+\frac{1}{n}\right) \\
     &\leq\frac{(\beta-\alpha+\frac{2d}{n} )(1+dC_2)+d(5C_1+\frac{2}{n})}{n}.\qedhere
\end{align*}
\end{proof}

%%%%%%%%%%%%%%%%%%%%%%%%%%%%%%%%%%%%%%%%
\subsection{A simplification}\label{Subsection:Simple}
If $\vec{x} = (x_1,x_2,x_3)$ and $\lambda( \vec{x}) \in \Lambda\multi{n}$, then
\begin{align*}
\lambda(\vec{x}) 
&\geq m_1 x_1 + m_2 x_2 + m_3 x_3
= \frac{m_1}{n_1} n_1 x_1 + \frac{m_2}{n_2} n_2 x_2 + \frac{m_3}{n_3}n_3 x_3 \\
&> \frac{m_3}{n_3}(n_1 x_1 + n_2 x_2 + n_3 x_3) = \frac{m_3}{n_3}n.
\end{align*}
Thus, $\Lambda\multi{n}\cap [-\infty,\tfrac{m_3}{n_3}n] = \varnothing$ and, in a similar manner, 
$\Lambda\multi{n}\cap [\tfrac{m_1}{n_1}n,\infty]=\varnothing$.
Since $F$ is supported on $[\frac{m_3}{n_3},\frac{m_1}{n_1}]$,
we may assume that $[\alpha,\beta]\in [\frac{m_3}{n_3},\frac{m_1}{n_1}]$. 
In particular, we can assume that the function $\ell(x,1)$ of Lemma \ref{Lemma:Triangle} is Lipschitz on $[\alpha,\beta]$.

%%%%%%%%%%%%%%%%%%%%%%%%%%%%%%%%%%%%%%%%
\subsection{Conclusion}
We now conclude the proof of Theorem \ref{Theorem:Main}. 
Fix $n \in \Z_{\geq 0}$ and let 
\begin{equation*}
f(m) = |\ZZ(m,n)| = \{ \vec{x} \in \ZZZ_S(n) : \lambda(\vec{x}) = m\}
\qquad \text{and} \qquad 
g(x)= \frac{ \ell(x,1) }{ \norm{ \vec{r} } }.
\end{equation*}
Let $d=\gcd(\rho_1,\rho_2,\rho_3)$ and deduce from \eqref{eq:Scaling} and Lemmas \ref{Lemma:whenEmpty} and \ref{Lemma:LineCount} 
that there is an $c \in \{0,1,\ldots,d-1\}$ such that
\begin{align*}
\left| \frac{f(c+kd)}{d}-ng\left(\frac{c+kd}{n}\right)\right|
&=\left| \frac{|\ZZ(c+kd,n)|}{d} -  \frac{n\ell(\frac{c+kd}{n},1)}{\norm{\vec{r}}}  \right| \\
&=\left| \frac{|\ZZ(c+kd,n)|}{d}-\frac{\ell(c+kd ,n)}{\norm{\vec{r}}} \right| \\
&\leq 1
\end{align*}
for all $k\in \Z$; moreover, $f(x)=0$ if $x\not \equiv c \pmod{d}$.

Suppose that $\rho_1,\rho_3 \neq 0$. Apply Lemma \ref{Lemma:Convergence} 
to the functions $f$ and $g$ and the parameters $c,d,n$ defined above, and to 
the constants
\begin{equation*}
C_1 = \frac{1}{n_2  \rho_2}
\qquad \text{and} \qquad
C_2 = \frac{1}{\rho_2} \operatorname{max}\left\{ \frac{n_3}{\rho_1}, \frac{n_1}{\rho_3} \right\}
\end{equation*}
provided by Lemma \ref{Lemma:Triangle}:
\begin{equation*}
\left| \frac{1}{n^2}\sum_{m \in \Z\cap [\alpha n, \beta n]}f(m)  - \int_{\alpha}^{\beta} g(x)\,dx \right|
\leq \frac{(\beta-\alpha +\frac{2d}{n})(1+dC_2)+d(5C_1+\frac{2}{n})}{n}.
\end{equation*}
Since
\begin{equation*}
\sum_{m \in \Z\cap [\alpha n, \beta n]} f(m)
=\big| \Lambda \multi{n} \cap  [\alpha n ,\beta n] \big| 
\end{equation*} 
and $\rho_1,\rho_2,\rho_3 \geq 1$, it follows that
\begin{align}
&\left |\frac{\big| \Lambda \multi{n} \cap  [\alpha n ,\beta n] \big|  }{n^2} -\int_{\alpha}^\beta \frac{\ell(x,1)}{\norm{\vec{r}}}\,dx \right | \nonumber\\
&\hspace{0.5in}\leq  \frac{(\beta-\alpha+\frac{2d}{n} )(1+dC_2)+d(5C_1+\frac{2}{n})}{n} \nonumber\\
&\hspace{0.5in} \leq \frac{(\beta-\alpha +\frac{2d}{n})(1+\frac{d}{\rho_2}\operatorname{max}\{ \frac{n_3}{\rho_1}, \frac{n_1}{\rho_3} \})+d(\frac{5}{n_2\rho_2}+\frac{2}{n})}{n} \label{eq:Error} \\
&\hspace{0.5in} \leq  \frac{(\beta-\alpha +\frac{2d}{n})(1+d\operatorname{max}\{ n_1,n_3 \})+d(\frac{5}{n_2}+\frac{2}{n})}{n} .\nonumber
\end{align}
To complete the proof of Theorem \ref{Theorem:Main} in this case, 
multiply by $2n_1n_2n_3$ and use \eqref{eq:Fell}. 
If $\rho_1=0$ or $\rho_3 = 0$,
the corresponding term in the maximum in \eqref{eq:Error} is omitted
by virtue of Remark \ref{Remark:Lipschitz} and the restriction of $[\alpha,\beta]$ in Subsection \ref{Subsection:Simple}. \qed

\begin{remark}\label{Remark:Error}
The bound implied by \eqref{eq:Error} is better, but more complicated, than the bound in Theorem \ref{Theorem:Main}.
The two bounds are compared in Table \ref{Table:errortab}.
\end{remark}

%%%%%%%%%%%%%%%%%%%%%%%%%%%%%%%%%%%%%%%%%%%%%%%%%%%%%%%%
%%%%%%%%%%%%%%%%%%%%%%%%%%%%%%%%%%%%%%%%%%%%%%%%%%%%%%%%
%%%%%%%%%%%%%%%%%%%%%%%%%%%%%%%%%%%%%%%%%%%%%%%%%%%%%%%%
%\appendix
%%%%%%%%%%%%%%%%%%%%%%%%%%%%%%%%%%%%%%%%%%%%%%%%%%%%%%

\bibliography{3-generator}
\bibliographystyle{amsplain}

\end{document}